\documentclass[11pt,reqno]{amsart}

\usepackage[T1]{fontenc}
\usepackage[utf8]{inputenc}
\usepackage{lmodern}
\usepackage{microtype}
\usepackage[margin=1in]{geometry}
\usepackage{amsmath,amssymb,mathtools}
\usepackage{enumitem}
\usepackage{xcolor}
\usepackage{hyperref}
\usepackage{aliascnt}
\usepackage[nameinlink,capitalise,noabbrev]{cleveref}

\hypersetup{
  colorlinks=true,
  linkcolor=blue!45!black,
  citecolor=magenta,
  urlcolor=blue!45!black,
  pdftitle={Fixed-particle-number optimizers for the Lieb--Oxford inequality},
  pdfauthor={Matthew Rosenzweig}
}

\allowdisplaybreaks
\setlist[itemize]{topsep=4pt,itemsep=2pt,parsep=1pt}
\setlist[enumerate]{topsep=4pt,itemsep=3pt,parsep=1pt}
\numberwithin{equation}{section}

\theoremstyle{plain}
\newtheorem{theorem}{Theorem}[section]
\newaliascnt{proposition}{theorem}
\newtheorem{proposition}[proposition]{Proposition}
\aliascntresetthe{proposition}
\newaliascnt{lemma}{theorem}
\newtheorem{lemma}[lemma]{Lemma}
\aliascntresetthe{lemma}
\newaliascnt{corollary}{theorem}

\aliascntresetthe{corollary}

\theoremstyle{definition}
\newaliascnt{definition}{theorem}
\newtheorem{definition}[definition]{Definition}
\aliascntresetthe{definition}

\theoremstyle{remark}
\newaliascnt{remark}{theorem}
\newtheorem{remark}[remark]{Remark}
\aliascntresetthe{remark}

\newcommand{\s}{{\mathsf{s}}}
\renewcommand{\d}{{\mathsf{d}}}
\newcommand{\diff}{\,\mathrm{d}}
\newcommand{\R}{\mathbb R}
\newcommand{\N}{\mathbb N}
\newcommand{\Pcal}{\mathcal P}

\newcommand{\Lcal}{\mathcal L}
\newcommand{\Wcal}{\mathcal W}
\newcommand{\Gcal}{\mathcal G}

\newcommand{\Acal}{\mathcal A}
\newcommand{\cLO}{c_{\mathrm{LO}}}
\newcommand{\supp}{\operatorname{supp}}
\newcommand{\Var}{\operatorname{Var}}
\newcommand{\E}{\mathbb E}

\newcommand{\norm}[1]{\left\lVert #1\right\rVert}
\newcommand{\one}{\mathbf 1}
\newcommand{\weak}{\rightharpoonup}
\newcommand{\weakstar}{\stackrel{\ast}{\rightharpoonup}}
\newcommand{\Sym}{\operatorname{Sym}}

\title[Fixed-Particle-Number Optimizers for the Lieb--Oxford Inequality]{Fixed-Particle-Number Optimizers for the Lieb--Oxford Inequality}
\author{Matthew Rosenzweig}
\address{Matthew Rosenzweig, Department of Mathematical Sciences, Carnegie Mellon University, 7127 Wean Hall, 5000 Forbes Avenue, Pittsburgh, PA 15213, USA}
\email{mrosenz2@andrew.cmu.edu}
\urladdr{https://matthewrosenzweigwork-max.github.io/}
\thanks{M.R. was supported by NSF grants DMS-2441170, DMS-2345533, and DMS-2342349.}
\begin{document}
\raggedbottom

\begin{abstract}
Let $\d\geq1$, $0<\s<\d$, and $N\geq1$.  We prove that the optimal fixed-particle-number constant $\Lambda_N(\s,\d)$ in the Riesz Lieb--Oxford inequality is attained and that these constants are strictly increasing in $N$.  The proof combines grand-canonical concentration--compactness with a strict one-particle extension.  After recentering, a limiting plan arising from a maximizing sequence may assign positive probability to several particle numbers and hence be grand-canonical.  A strict $N$-particle completion excludes this case, while the inequalities $\Lambda_N>\Lambda_k$ for $k<N$ exclude limits with a fixed lower particle number.  Once attainment at particle number $N$ is known, the compact-support theorem of Di Marino and Lelotte \cite{DiMarinoLelotte2026}, valid for all $0<\s<\d$, permits a non-product one-particle extension and yields $\Lambda_{N+1}(\s,\d)>\Lambda_N(\s,\d)$.  Together, these implications close an induction beginning at $N=1$.
\end{abstract}

\maketitle

\section{Introduction}\label{sec:intro}

\subsection{The fixed-particle-number problem and previous work}

Let $\d\geq1$ be an integer and let $0<\s<\d$.  For $N\in\N$, let $P$ be a symmetric probability measure on $(\R^\d)^N$, and, for $1\leq i\leq N$, let $\pi_i:(\R^\d)^N\to\R^\d$ denote the $i$-th coordinate projection.  Then
\begin{equation}\label{eq:rhoP-intro}
  \rho_P=\sum_{i=1}^N(\pi_i)_\#P,
  \qquad \rho_P(\R^\d)=N,
\end{equation}
is its one-body measure.  When this measure is absolutely continuous with respect to Lebesgue measure, we use the same notation $\rho_P$ for its density.  For such $P$, the indirect Riesz energy of $P$, denoted by $\mathcal E_{\mathrm{ind}}[P]$, obeys
\begin{align}
  \mathcal E_{\mathrm{ind}}[P]
  &:=\int_{(\R^\d)^N}\sum_{1\leq i<j\leq N}\frac1{|x_i-x_j|^\s}\diff P
  -\frac12\iint_{\R^\d\times\R^\d}
  \frac{\rho_P(x)\rho_P(y)}{|x-y|^\s}\diff x\diff y \notag\\
  &\geq-\cLO(\s,\d)\int_{\R^\d}\rho_P^{1+\s/\d}.
  \label{eq:Eind-intro}
\end{align}
Lieb first proved this inequality for the three-dimensional Coulomb interaction, corresponding to $(\s,\d)=(1,3)$ \cite{Lieb1979}.  Lieb and Oxford subsequently improved the universal constant \cite{LiebOxford1981}.  The corresponding power-law Riesz inequality is valid for every $\d\geq1$ and $0<\s<\d$; see \cite[Section~5.3]{LewinLiebSeiringer2023} and references therein.  Here, $\cLO(\s,\d)<\infty$ is independent of $N$.

For a density $\rho\geq0$ of mass $N$, the repulsive interaction at fixed one-body density is the multimarginal optimal-transport problem
\begin{equation}\label{eq:WN-intro}
  \Wcal_N(\rho)
  :=\min\left\{
  \int_{(\R^\d)^N}\sum_{1\leq i<j\leq N}\frac1{|x_i-x_j|^\s}\diff P:
  P\in\Pcal_{\rm sym}((\R^\d)^N),\ \rho_P=\rho
  \right\}.
\end{equation}
The rigorous theory of multimarginal transport with Coulomb and more general repulsive singular costs was developed in \cite{CotarFrieseckeKlueppelberg2013,DePascale2015,ButtazzoChampionDePascale2018,ColomboDiMarinoStra2019}; see also \cite{SeidlBenyahiaKooiGoriGiorgi2022} for an account directed toward density-functional theory.  Set
\begin{equation}\label{eq:p-intro}
  p:=1+\frac{\s}{\d},
\end{equation}
and define
\begin{equation}\label{eq:GN-intro}
  \Gcal_N(\rho):=\frac12D(\rho)-\Wcal_N(\rho),
  \qquad
  D(\rho):=\iint_{\R^\d\times\R^\d}\frac{\rho(x)\rho(y)}{|x-y|^\s}\diff x\diff y.
\end{equation}
By the definitions of $\Wcal_N$ and $\mathcal E_{\mathrm{ind}}$, $\inf_{\rho_P=\rho}\mathcal E_{\mathrm{ind}}[P]=-\Gcal_N(\rho)$.
The corresponding sharp constant is\footnote{We write $\Lambda_N(\s,\d)$ for the fixed-particle-number constant and reserve $\cLO(\s,\d)$ for the universal constant.  When $(\s,\d)=(1,3)$, $\Lambda_N(1,3)=C_N$ in Lieb and Oxford's notation.  Following Di Marino and Lelotte \cite{DiMarinoLelotte2026}, $D(\rho)$ denotes the full double integral, so the Hartree term is $\frac12D(\rho)$.  The restriction to symmetric plans entails no loss, since symmetrization preserves both the one-body density and the pair cost.}
\begin{equation}\label{eq:LambdaN-intro}
  \Lambda_N=\Lambda_N(\s,\d):=\sup_{\rho\in\Acal_N}\frac{\Gcal_N(\rho)}{\Lcal(\rho)},
  \qquad
  \Lcal(\rho):=\int_{\R^\d}\rho^p,
\end{equation}
where
\begin{equation}\label{eq:AN-intro}
  \Acal_N:=\left\{\rho\in L^1(\R^\d)\cap L^p(\R^\d):
  \rho\geq0,\ \int_{\R^\d}\rho=N\right\}.
\end{equation}
Here and below, an \emph{optimizer} is a density $\rho\in\Acal_N$ attaining the supremum in \eqref{eq:LambdaN-intro}, whereas an \emph{optimal plan} is a minimizer in \eqref{eq:WN-intro}.
Throughout the proof, $\s$ and $\d$ are fixed and we abbreviate $\Lambda_N(\s,\d)$ to $\Lambda_N$.  The universal Riesz constant is
\begin{equation}
  \cLO(\s,\d)=\sup_{N\geq1}\Lambda_N(\s,\d).
\end{equation}

In the three-dimensional Coulomb case, Gadre, Bartolotti, and Handy analyzed the variational problem defining the sharp one-particle constant $\Lambda_1(1,3)$, observed that its Euler--Lagrange equation is the Lane--Emden equation of order three, and obtained $\Lambda_1(1,3)\approx1.0918$ \cite{GadreBartolottiHandy1980}.  Lieb and Oxford subsequently proved that the supremum is attained by a symmetric decreasing density, established compact support of the resulting optimizer, and gave a rigorous derivation of the Lane--Emden characterization \cite[Appendix~A]{LiebOxford1981}.  Subsequent work improved the universal upper bound and developed the density-functional implications of the inequality \cite{ChanHandy1999,LewinLiebSeiringer2022,PerdewSun2022}.  Determining these sharp constants remains an important open problem.  Even in this case, the exact value of $\cLO(1,3)$ is unknown: a widely discussed candidate is approximately $1.4442$, obtained by conjecturally identifying it with the body-centered-cubic uniform-electron-gas value \cite{CotarPetrache2019,LewinLiebSeiringer2022}; beyond the one-particle Lane--Emden characterization, the values $\Lambda_N(1,3)$ likewise remain undetermined to our knowledge.  The sharp fixed-particle-number constants $\Lambda_N(1,3)$ have also been studied in connection with particle-number-dependent constraints and tests of the tightness of the universal inequality; see \cite{OdashimaCapelle2007,OdashimaCapelleTrickey2009,SeidlVuckovicGoriGiorgi2016,SeidlBenyahiaKooiGoriGiorgi2022,PerdewSun2022}.  For $N\geq2$, these works provide numerical bounds, structural reformulations, and guidance for approximate density functionals, but, to our knowledge, do not prove attainment of the exact fixed-$N$ variational problem.

By adding a particle whose position has a smooth probability density dilated to an increasingly large spatial scale (a \emph{diffuse particle at infinity}), Lieb and Oxford proved \cite[Appendix~B]{LiebOxford1981}, in the Coulomb case,
\begin{equation}
  \Lambda_N\leq\Lambda_{N+1}.
\end{equation}
The argument uses only homogeneity and therefore applies throughout the full Riesz range $0<\s<\d$; we recall it in \Cref{prop:monotonicity}.\footnote{\cite[Eq.~(16)]{SeidlVuckovicGoriGiorgi2016} prints a strict inequality and attributes it to Lieb and Oxford, but the cited primary argument is nonstrict; the same group's later account likewise states the nonstrict form \cite[Eq.~(1.12)]{SeidlBenyahiaKooiGoriGiorgi2022}.}

The existence question is different.  Seidl, Vuckovic, and Gori-Giorgi gave a formal argument suggesting that any Coulomb optimizer, if one exists, must be compactly supported \cite{SeidlVuckovicGoriGiorgi2016}.  Di Marino and Lelotte made the compact-support conclusion rigorous throughout the full Riesz range $0<\s<\d$: every fixed-particle optimizer is compactly supported if it exists \cite{DiMarinoLelotte2026}.  Their proof uses the duality and potential theory developed in \cite{DePascale2015,ButtazzoChampionDePascale2018,ColomboDiMarinoStra2019,Lelotte2024,Lelotte2025}.  Their result does not itself address attainment.  At $N=1$, however, Calvez, Carrillo, and Hoffmann had already proved attainment throughout the full Riesz range and obtained a radially non-increasing, compactly supported, bounded optimizer \cite[Theorem~2.8, Proposition~3.4, and Remark~3.5]{CalvezCarrilloHoffmann2017}.  Thus, prior to the present work, the fixed-particle attainment problem remained open for $N\geq2$.

For $N\geq2$, the remaining task is therefore to rule out loss of compactness along maximizing sequences.  Translation and dilation invariance allow such a sequence to split into widely separated pieces.  After recentering around one piece, the limiting $N$-particle plan may assign positive mass to several particle-number sectors, according to how many coordinates remain in bounded sets.  We call the resulting sector law, together with its one-body density, a profile.  Its mean particle number is the expected number of coordinates that remain finite after recentering, equivalently the mass of its one-body density.  The weak limit of the one-body densities records the combined local density but not its distribution among these sectors.  We must therefore retain this sector information in the limit and prove that a maximizing sequence cannot have only profiles with mean particle number strictly below $N$.

\subsection{Main result}

The following theorem resolves this compactness problem for every $N\geq2$, recovers the known $N=1$ result within a unified argument, and shows that the preceding monotonicity is strict.

\begin{theorem}[Main theorem]\label{thm:main}
Let $\d\geq1$ and $0<\s<\d$.  For every $N\geq1$, the supremum in \eqref{eq:LambdaN-intro} is attained.  Moreover,
\begin{equation}\label{eq:strict-sequence}
  0<\Lambda_1(\s,\d)<\Lambda_2(\s,\d)<\Lambda_3(\s,\d)<\cdots.
\end{equation}
Consequently, every fixed-$N$ optimizer is compactly supported.
\end{theorem}

Since $\cLO(\s,\d)=\sup_{M\geq1}\Lambda_M(\s,\d)$, the strict inequalities in \eqref{eq:strict-sequence} imply that the universal Riesz Lieb--Oxford constant is not attained at any finite particle number.

At the one-particle level, the multimarginal interaction cost vanishes, so \eqref{eq:LambdaN-intro} reduces to the scale-invariant Hartree/local-power quotient over unit-mass densities.  If we write the spatial dimension in Calvez, Carrillo, and Hoffmann as our $\d$ and set $k=-\s$ and $m=1+\s/\d$, then their best constant $C_*$ in \cite[Eq.~(3.2)]{CalvezCarrilloHoffmann2017} satisfies $C_*=2\Lambda_1(\s,\d)$; the factor $2$ comes from the factor $1/2$ in our definition of the Hartree term.  The concentration--compactness argument developed below therefore gives an independent proof of the known one-particle attainment result and supplies the base case of the induction.

\subsection{Method of proof and relation to previous work}

The proof is organized around three implications:
\begin{align}
  \Lambda_k&\leq\Lambda_N \qquad \forall\,1\leq k\leq N,
  \label{eq:induction-monotonicity}\\
  \Lambda_N>\max_{1\leq k<N}\Lambda_k
  &\quad\Longrightarrow\quad \Lambda_N\text{ is attained},
  \label{eq:induction-attainment}\\
  \Lambda_N\text{ is attained}
  &\quad\Longrightarrow\quad \Lambda_{N+1}>\Lambda_N.
  \label{eq:induction-extension}
\end{align}
The maximum in \eqref{eq:induction-attainment} is absent when $N=1$, which gives the base case.  Once $\Lambda_N$ is attained, \eqref{eq:induction-extension} creates a strict gap at $N+1$.  By \eqref{eq:induction-monotonicity}, that gap lies above every lower-particle constant, and \eqref{eq:induction-attainment} gives attainment at $N+1$.  Thus, these implications form a closed induction.

The compactness implication \eqref{eq:induction-attainment} is of Hunziker--van Winter--Zhislin (HVZ) type (cf. \cite[Section~1]{Lewin2011}) in the following precise sense: if compactness fails, the profile decomposition yields a family of profiles, each with mean particle number strictly below $N$.  A deterministic profile then belongs to an exact $k$-particle sector with $k<N$ and is excluded by the strict inequality $\Lambda_N>\Lambda_k$; a profile with fluctuating particle number is excluded by the strict finite-$N$ completion argument below.

At the density level, we use concentration--compactness, following Lions and later refinements by Solimini, G\'erard, and Jaffard \cite{Lions1984I,Lions1984II,Solimini1995,Gerard1998,Jaffard1999}.  Given the densities of a maximizing sequence, we recenter around one concentration of mass, extract its limiting density, and repeat the procedure on what remains.  This produces at most countably many spatially separated density profiles that together account for all of the Riesz--Hartree energy.

At the plan level, after recentering around a density profile, only some particles may remain nearby; the others escape to infinity.  Following Bouchitt\'e, Buttazzo, Champion, and De Pascale \cite{BouchitteEtAl2021Relaxed,BouchitteEtAl2021Dissociation}, we record the escaped particles in the one-point compactification.  Decomposing by the number of particles that remain finite yields, in the terminology of Di Marino, Lewin, and Nenna \cite{DiMarinoLewinNenna2025}, a truncated grand-canonical probability $\Gamma=(\Gamma^{(0)},\ldots,\Gamma^{(N)})$.  Its particle-number random variable is defined by $K=k$ on the $k$-particle sector; thus, $\mathbb P(K=k)$ is the total mass of $\Gamma^{(k)}$.  The associated density profile $\rho_\Gamma$ is the sector-weighted sum of the one-body densities and has total mass $\E[K]$; see \Cref{def:grand-canonical} and \eqref{eq:grand-canonical-density}.  Applying this at every spatial profile matches the plan and density decompositions and bounds the sum of the profile pair costs by the limiting pair cost; see \Cref{prop:profile-decomposition}.

To exclude a nonzero profile with $\E[K]<N$, we show that its quotient $\Gcal(\Gamma)/\Lcal(\rho_\Gamma)$ is strictly below $\Lambda_N$.  If $K=k$ almost surely, this follows from $\Lambda_k<\Lambda_N$.  If $K$ fluctuates, we complete each $k$-particle sector to exactly $N$ particles by adding $N-k$ particles and symmetrizing.  Related canonicalization constructions for grand-canonical Coulomb and Riesz states appear in Lewin, Lieb, and Seiringer \cite[Lemma~3.2]{LewinLiebSeiringer2018}; our completion is performed directly at the prescribed finite particle number.  The identity in \Cref{lem:completion-identity} turns $\Var(K)>0$ into a signed interaction defect of positive total mass.  Placing a small part of the added density on a remote annulus detects this defect and gives a first-order gain, while placing the rest on a much larger scale makes the remaining costs lower order because $p>1$.  The completed canonical state therefore has a strictly larger quotient, which is the strict comparison in \Cref{prop:strict-completion}.

For the implication \eqref{eq:induction-extension}, suppose $\rho$ attains $\Lambda_N$ and let $P$ be an associated optimal plan.  By the compact-support theorem of Di Marino and Lelotte \cite{DiMarinoLelotte2026}, $\rho$ is compactly supported, hence $P$ is supported on configurations in a fixed compact set; see \Cref{prop:DML}.  Adding an independent diffuse particle at infinity, as in \Cref{prop:monotonicity}, yields only a nonstrict comparison.  For strictness, we couple a small exterior portion of its density non-independently with the original configuration, preserving both marginals.  If the product coupling were optimal, the two-cycle condition \cite{SmithKnott1992,GangboMcCann1996,Ruschendorf1996} would force the exterior Riesz potentials of any two configurations in the support of $P$ to differ only by a constant.  The exterior uniqueness provided by \Cref{lem:exterior-injectivity} would then force their empirical counting measures to be identical, which is impossible because their average $\rho$ is absolutely continuous.  The perturbation in \Cref{lem:strict-coupling} therefore lowers the cross-interaction between the original $N$-particle configuration and the added particle.  Sending the remaining added-particle mass far away makes the local-power cost lower order and yields $\Lambda_{N+1}>\Lambda_N$; see \Cref{prop:strict-extension}.

The proof draws on several established ingredients: concentration--compactness, compactification into lower-particle sectors, grand-canonical formalism, the HVZ-type perspective, the compact-support theorem of Di Marino and Lelotte, and the general cyclical-monotonicity criterion.  To our knowledge, the main new elements are two finite-particle comparison arguments.  The first combines an exact completion identity at a prescribed particle number with a two-scale added density to obtain a strict comparison for profiles with fluctuating particle number.  The second uses compact support together with a non-product coupling to prove strict monotonicity of the fixed-particle Lieb--Oxford constants.

% \begin{remark}[Fock-space interpretation]\label{rem:fock-intro}
% A truncated grand-canonical probability may be viewed as the classical probabilistic analogue of a block-diagonal state on truncated Fock space.  This is only an analogy: the object used below is a probability measure on
% \[
%   \bigsqcup_{k=0}^N(\R^\d)^k/\mathfrak S_k,
% \]
% with no off-diagonal coherences, creation or annihilation operators, or Bose--Fermi structure beyond ordinary permutation symmetry within each sector.
% \end{remark}

\subsection{Organization of the paper}

\Cref{sec:setup} records the fixed-particle variational problem, the basic Riesz estimates, and the Lieb--Oxford monotonicity argument.  \Cref{sec:profiles} constructs the grand-canonical limits of maximizing sequences; the density-profile extraction and Riesz--Hartree estimates are placed in \Cref{app:density-profiles}.  \Cref{sec:grand-canonical} proves the strict completion of a profile with fluctuating particle number, and \Cref{sec:attainment} derives compactness under the strict lower-particle inequality.  \Cref{sec:one-particle-extension} proves the strict extension from $N$ to $N+1$.  The induction is closed in \Cref{sec:conclusion}, which also discusses consequences and outlook.

\subsection{Statement on AI use}

The author used generative AI tools (OpenAI's ChatGPT and Codex) during the development and preparation of this paper to identify potentially relevant literature, explore and test mathematical arguments, improve the exposition, and check internal consistency and cross-references.  The author treated all AI-generated output as provisional material requiring independent verification and did not rely on it as an authority, checking literature suggestions against the relevant sources and independently working through mathematical suggestions before deciding whether to incorporate them into the paper.  The author made all final decisions concerning the manuscript and takes full responsibility for its contents.

\section{The fixed-particle-number problem}\label{sec:setup}

\subsection{Optimal plans, scaling, and finiteness}

For nonnegative functions $f,g$, write
\begin{equation}\label{eq:Dfg}
  D(f,g):=\iint_{\R^\d\times\R^\d}\frac{f(x)g(y)}{|x-y|^\s}\diff x\diff y,
  \qquad D(f):=D(f,f).
\end{equation}
If $f\in L^1(\R^\d)$ is real-valued and $g\geq0$ is such that
\begin{equation}
  \int_{\R^\d}|f(x)|U_g(x)\diff x<\infty,
  \qquad
  U_g(x):=\int_{\R^\d}\frac{g(y)}{|x-y|^\s}\diff y,
\end{equation}
we also write
\begin{equation}
  D(f,g):=\int_{\R^\d} f(x)U_g(x)\diff x.
\end{equation}
This signed pairing is absolutely convergent.

\begin{lemma}[Optimal plans]\label{lem:OT-existence}
For every $\rho\in\Acal_N$, the minimum in \eqref{eq:WN-intro} is attained.
\end{lemma}

\begin{proof}
The product probability $(\rho/N)^{\otimes N}$ is admissible and has finite cost because the Hardy--Littlewood--Sobolev inequality implies $D(\rho)<\infty$.  The set of probabilities with the prescribed marginals is tight and weakly closed, hence weakly compact.  The Riesz pair cost is nonnegative and lower semicontinuous.  The direct method gives a minimizer (see, e.g., \cite[Theorem~4.1 and Lemmas~4.3--4.4]{Villani2009}), and symmetrization preserves both the marginals and the cost.
\end{proof}

Set
\begin{equation}\label{eq:HLS-exponent}
  \sigma:=\frac{2\d}{2\d-\s}.
\end{equation}
Then $1<\sigma<p$, and the Hardy--Littlewood--Sobolev inequality (see \cite{Lieb1983} for the sharp form) together with interpolation gives
\begin{equation}\label{eq:HLS-interpolation}
  D(f,g)\leq C_{\s,\d}\norm{f}_{\sigma}\norm{g}_{\sigma},
  \qquad
  \norm{f}_{\sigma}
  \leq\norm{f}_1^{\frac{\d-\s}{2\d}}
  \norm{f}_p^{\frac{\d+\s}{2\d}}.
\end{equation}
In particular, recalling that $\Lcal(f)=\int_{\R^\d}f^p$ as in \eqref{eq:LambdaN-intro},
\begin{equation}\label{eq:D-mass-L}
  D(f)\leq C_{\s,\d}\norm{f}_1^{1-\s/\d}\Lcal(f),
\end{equation}
for every nonnegative $f\in L^1\cap L^p$.  We use these basic estimates repeatedly in the sequel.

For $\lambda>0$, set
\begin{equation}\label{eq:scaling}
  \rho_\lambda(x)=\lambda^\d\rho(\lambda x).
\end{equation}
Then
\begin{equation}\label{eq:scaling-all}
  \begin{aligned}
    \int_{\R^\d}\rho_\lambda&=\int_{\R^\d}\rho,
    & D(\rho_\lambda)&=\lambda^sD(\rho),\\
    \Wcal_N(\rho_\lambda)&=\lambda^\s\Wcal_N(\rho),
    & \Lcal(\rho_\lambda)&=\lambda^\s\Lcal(\rho).
  \end{aligned}
\end{equation}
Thus, the quotient in \eqref{eq:LambdaN-intro} is invariant under translations and mass-preserving dilations.

Since $\Wcal_N\geq0$, \eqref{eq:D-mass-L} gives
\begin{equation}\label{eq:Lambda-finite}
  \Lambda_N\leq C_{\s,\d}N^{1-\s/\d}<\infty.
\end{equation}
Also, $\Lambda_1>0$, because $\Wcal_1=0$ and $D(\rho)>0$ for every nonzero admissible density.

\subsection{Nonstrict monotonicity}

The following proposition is proved by the diffuse-particle argument of Lieb and Oxford \cite{LiebOxford1981}.  By homogeneity, the same argument applies throughout the stated Riesz range.

\begin{proposition}[Monotonicity]\label{prop:monotonicity}
If $1\leq k\leq N$, then
\begin{equation}\label{eq:monotonicity}
  \Lambda_k\leq\Lambda_N.
\end{equation}
\end{proposition}

\begin{proof}
Assume $k<N$; otherwise, the assertion is tautological.  Let $\rho\in\Acal_k$ and let $P$ be an optimal $k$-particle plan.  Choose a smooth probability density $\eta$ and define $\eta_R(x)=R^{-\d}\eta(x/R)$.  Form the product $P\otimes\eta_R^{\otimes (N-k)}$ and symmetrize it over the $N$ labels.  The resulting canonical $N$-particle plan has density $\rho+(N-k)\eta_R$ and pair cost
\begin{equation}
  \Wcal_k(\rho)+(N-k)D(\rho,\eta_R)+\binom{N-k}{2}D(\eta_R).
\end{equation}
Consequently,
\begin{align}\label{eq:monotonicity-gain}
  \Gcal_N(\rho+(N-k)\eta_R)
  &\geq
  \frac12D(\rho+(N-k)\eta_R)
  -\left[
  \Wcal_k(\rho)+(N-k)D(\rho,\eta_R)
  +\binom{N-k}{2}D(\eta_R)
  \right]\notag\\
  &=\Gcal_k(\rho)+\frac{N-k}2D(\eta_R).
\end{align}
Since $D(\eta_R)=R^{-\s}D(\eta)\to0$ and
\begin{equation}
  \norm{(N-k)\eta_R}_p^p=(N-k)^pR^{-\s}\Lcal(\eta)\longrightarrow0,
\end{equation}
we have $\Lcal(\rho+(N-k)\eta_R)\to\Lcal(\rho)$.  Taking quotients and then the supremum over $\rho\in\Acal_k$ gives $\Lambda_k\leq\Lambda_N$.
\end{proof}

\section{Grand-canonical limits of maximizing sequences}\label{sec:profiles}

A normalized maximizing sequence need not be tight modulo a single translation.  At the density level, it may split into countably many profiles.  At the plan level, translating to one profile can leave only some of the $N$ coordinates in bounded sets.  The correct local object is therefore a truncated grand-canonical probability.  \Cref{app:density-profiles} supplies the density profile decomposition and the exact Hartree identity.  The present section identifies the simultaneous plan limits and proves the lower-semicontinuity estimate for their pair costs.

Following Di Marino, Lewin, and Nenna \cite{DiMarinoLewinNenna2025}, we call a probability concentrated in one particle-number sector \emph{canonical} and a probability on the disjoint union of sectors, constrained through its first moment, \emph{grand-canonical}.  All grand-canonical probabilities below are truncated to the sectors $0,\ldots,N$.

\subsection{Truncated grand-canonical probabilities}

\begin{definition}[Truncated grand-canonical probability]\label{def:grand-canonical}
A truncated grand-canonical probability with sectors at most $N$ is a tuple
\begin{equation}
  \Gamma=(\Gamma^{(0)},\Gamma^{(1)},\ldots,\Gamma^{(N)}),
\end{equation}
where $\Gamma^{(0)}\geq0$ is a scalar, each $\Gamma^{(k)}$ for $k\geq1$ is a finite symmetric nonnegative measure on $(\R^\d)^k$, $\norm{\Gamma^{(k)}}$ denotes the total mass of $\Gamma^{(k)}$, and
\begin{equation}\label{eq:grand-canonical-normalization}
  \Gamma^{(0)}+\sum_{k=1}^N\norm{\Gamma^{(k)}}=1.
\end{equation}
\end{definition}

Let $K$ be the random variable with law
\begin{equation}
  \mathbb P(K=0):=\Gamma^{(0)},
  \qquad
  \mathbb P(K=k):=\norm{\Gamma^{(k)}}\quad(1\leq k\leq N).
\end{equation}
For $1\leq k\leq N$, the one-body measure of the $k$-particle sector is
\begin{equation}\label{eq:sector-density}
  \rho_k:=\sum_{i=1}^k(\pi_i)_\#\Gamma^{(k)}
  =k(\pi_1)_\#\Gamma^{(k)},
  \qquad \rho_0:=0,
\end{equation}
where $\pi_i$ denotes the $i$-th coordinate projection.  Let
\begin{equation}\label{eq:grand-canonical-density}
  \rho_\Gamma:=\sum_{k=1}^N\rho_k,
  \qquad
  \int_{\R^\d}\rho_\Gamma=\E K.
\end{equation}
The grand-canonical average pair cost is $\Wcal(\Gamma)$.  We write $\Gcal(\Gamma)$ for the associated signed gain relative to the Hartree term:
\begin{align}
  \Wcal(\Gamma)
  &:=\sum_{k=2}^N
  \int_{(\R^\d)^k}\sum_{1\leq i<j\leq k}\frac1{|x_i-x_j|^\s}\diff\Gamma^{(k)},
  \label{eq:grand-canonical-cost}\\
  \Gcal(\Gamma)
  &:=\frac12D(\rho_\Gamma)-\Wcal(\Gamma).
  \label{eq:grand-canonical-gain}
\end{align}
Equivalently, $-\Gcal(\Gamma)$ is the indirect Riesz energy of $\Gamma$.
We only consider grand-canonical probabilities for which $\rho_\Gamma\in L^1\cap L^p$ and the quantities above are finite.  Since the sector measures are positive and sum to an absolutely continuous measure, our assumption implies that each $\rho_k$ is itself absolutely continuous.

\subsection{Plan limits and particle sectors}

Let $P_n\in\Pcal_{\rm sym}((\R^\d)^N)$ have one-body densities $\rho_n$ with
\begin{equation}
  \int_{\R^\d}\rho_n=N,
  \qquad
  \sup_n\Lcal(\rho_n)<\infty.
\end{equation}
Assume in addition that
\begin{equation}\label{eq:uniform-plan-pair-cost}
  \sup_n
  \int_{(\R^\d)^N}
  \sum_{1\leq i<j\leq N}
  \frac{1}{|x_i-x_j|^\s}\diff P_n
  <\infty.
\end{equation}
By \Cref{prop:density-profiles}, after passing to a subsequence, there are profiles $\rho^{(\alpha)}$ and translation centers $a_n^{(\alpha)}$ satisfying the center-separation condition \eqref{eq:center-separation} and the mass and local-energy bounds \eqref{eq:profile-budgets-density}.  Passing to the further subsequence furnished by \Cref{prop:Hartree-profile}, we also have
\begin{equation}
  \lim_{n\to\infty}D(\rho_n)=\sum_\alpha D(\rho^{(\alpha)}).
\end{equation}
We now lift these density profiles to the plans.

Following the relaxed-cost construction of Bouchitt\'e, Buttazzo, Champion, and De Pascale \cite{BouchitteEtAl2021Relaxed}, let
\begin{equation}
  X:=\R^\d\cup\{\infty\}
\end{equation}
be the one-point compactification.  Extend the Riesz pair cost to $X\times X$ by
\begin{equation}\label{eq:extended-cost}
  \overline c(x,y)
  :=\begin{cases}
  |x-y|^{-\s},&x,y\in\R^\d,\ x\neq y,\\
  +\infty,&x=y\in\R^\d,\\
  0,&x=\infty\text{ or }y=\infty.
  \end{cases}
\end{equation}
This function is nonnegative and lower semicontinuous.

For $b\in\R^\d$, let $\tau_b(x):=x+b$.  For each profile center, define
\begin{equation}\label{eq:translated-plans}
  Q_n^{(\alpha)}
  :=\left((\tau_{-a_n^{(\alpha)}})^{\otimes N}\right)_\#P_n,
\end{equation}
viewed as probability measures on the compact space $X^N$.  After a diagonalization argument, we may assume that
\begin{equation}\label{eq:plan-weak-limit}
  \forall\,\alpha,\qquad
  Q_n^{(\alpha)}\weakstar\widetilde P^{(\alpha)}
  \quad\text{in }\Pcal(X^N).
\end{equation}

For $1\leq k\leq N$, let $\pi_{1,\ldots,k}:X^N\to X^k$ denote the projection onto the first $k$ coordinates.  For $0\leq k\leq N$, define the $k$-particle sector of $\widetilde P^{(\alpha)}$ by
\begin{equation}\label{eq:sector-from-compactification}
  \Gamma^{(\alpha,k)}
  :=\binom Nk(\pi_{1,\ldots,k})_\#
  \left[
  \widetilde P^{(\alpha)}\big|_{(\R^\d)^k\times\{\infty\}^{N-k}}
  \right],
\end{equation}
with the evident scalar interpretation for $k=0$.  The binomial coefficient sums over the choices of the $k$ labels that remain finite.  This is the standard sector normalization for localized canonical plans and truncated grand-canonical probabilities \cite{BouchitteEtAl2021Relaxed,DiMarinoLewinNenna2025}.  Symmetry implies that these sector masses sum to one.  Let
\begin{equation}
  \Gamma^{(\alpha)}
  :=(\Gamma^{(\alpha,0)},\ldots,\Gamma^{(\alpha,N)}).
\end{equation}
For every profile index $\alpha$ and every $\varphi\in C_c(\R^\d)$, the profile convergence and support separation in \Cref{prop:density-profiles}, together with $\norm{e_n^{(J)}}_1\to0$ for fixed $J\geq\alpha$, give
\begin{equation}\label{eq:full-density-local-limit}
  \int_{\R^\d}\varphi(x)\rho_n(x+a_n^{(\alpha)})\diff x
  \longrightarrow
  \int_{\R^\d}\varphi(x)\rho^{(\alpha)}(x)\diff x.
\end{equation}
For the same $\alpha$ and $\varphi$,
\begin{equation}
  \int_{\R^\d}\varphi(x)\rho_n(x+a_n^{(\alpha)})\diff x = \int_{X^N}\sum_{i=1}^N\varphi(x_i)\diff Q_n^{(\alpha)} \longrightarrow \int_{X^N}\sum_{i=1}^N\varphi(x_i)\diff\widetilde P^{(\alpha)}.
\end{equation}
Comparison with \eqref{eq:full-density-local-limit} gives
\begin{equation}\label{eq:grand-canonical-profile-density}
  \rho_{\Gamma^{(\alpha)}}=\rho^{(\alpha)}.
\end{equation}

For a single translated limit, the compactified limiting plan splits into the sectors defined above: its $k$-particle sector is the part supported on configurations with exactly $k$ coordinates in Euclidean space and $N-k$ coordinates at infinity.  The following estimate is the corresponding lower-semicontinuity mechanism of \cite{BouchitteEtAl2021Relaxed}.  The additional point is that disjoint localization allows all selected translated profiles to be summed before passing to the countable limit.

\begin{proposition}[Profile-cost lower semicontinuity]\label{prop:cost-profile}
With the notation above,
\begin{equation}\label{eq:cost-profile}
  \sum_{\alpha=1}^\infty\Wcal(\Gamma^{(\alpha)})
  \leq\liminf_{n\to\infty}
  \int_{(\R^\d)^N}\sum_{1\leq i<j\leq N}
  \frac1{|x_i-x_j|^\s}\diff P_n.
\end{equation}
\end{proposition}

\begin{proof}
Choose radial continuous cutoffs $0\leq\chi_R\leq1$, pointwise nondecreasing in $R$, with $\chi_R=1$ on $B_R$, $\chi_R=0$ outside $B_{2R}$, and $\chi_R(\infty)=0$.  Define
\begin{equation}
  w_R(x,y):=
  \begin{cases}
    \chi_R(x)\chi_R(y)\,\overline c(x,y),
    &\chi_R(x)\chi_R(y)>0,\\
    0,&\chi_R(x)\chi_R(y)=0,
  \end{cases},
\end{equation}
and set on $X^N$
\begin{equation}\label{eq:FR}
  F_R(x_1,\ldots,x_N)
  :=\sum_{1\leq i<j\leq N} w_R(x_i,x_j).
\end{equation}
The function $w_R$, and hence $F_R$, is nonnegative and lower semicontinuous.  At a point where the cutoff product is zero, lower semicontinuity follows from nonnegativity.  At a point where the cutoff product is positive, it follows from continuity of that positive factor and lower semicontinuity of $\overline c$.  Fix $M\in\N$ and $R>0$.  By \eqref{eq:center-separation}, for all sufficiently large $n$, the balls
\begin{equation}
  B_{2R}(a_n^{(1)}),\ldots,B_{2R}(a_n^{(M)})
\end{equation}
are pairwise disjoint.  Fix a pair of indices $i<j$.  For a given profile $\alpha$, the term $w_R(x_i-a_n^{(\alpha)},x_j-a_n^{(\alpha)})$ can be nonzero only if both $x_i$ and $x_j$ lie in $B_{2R}(a_n^{(\alpha)})$.  Since these balls are pairwise disjoint, at most one profile index $\alpha$ can contribute for this fixed pair.  For that index, the cutoff factors are bounded by one, and translating both coordinates does not change their separation.  Hence, the localized contribution is bounded above by the original Riesz interaction between $x_i$ and $x_j$.  Summing first over the profile indices and then over all particle pairs gives the pointwise estimate \eqref{eq:localized-cost-sum}:
\begin{equation}\label{eq:localized-cost-sum}
  \sum_{\alpha=1}^M
  F_R(x_1-a_n^{(\alpha)},\ldots,x_N-a_n^{(\alpha)})
  \leq\sum_{1\leq i<j\leq N}\frac1{|x_i-x_j|^\s}.
\end{equation}
Portmanteau's theorem and \eqref{eq:plan-weak-limit} yield
\begin{equation}
  \sum_{\alpha=1}^M\int_{X^N} F_R\diff\widetilde P^{(\alpha)}
  \leq\liminf_n\int_{(\R^\d)^N}\sum_{i<j}\frac1{|x_i-x_j|^\s}\diff P_n.
\end{equation}
Monotone convergence and the sector decomposition give
\begin{equation}
  \int_{X^N} F_R\diff\widetilde P^{(\alpha)}
  \uparrow\Wcal(\Gamma^{(\alpha)})
  \qquad\text{as }R\to\infty.
\end{equation}
Finally, letting $M\to\infty$ completes the proof.
\end{proof}

The density profiles from \Cref{app:density-profiles}, their exact Hartree decoupling, and the preceding plan-level lower-semicontinuity estimate give the compactness statement needed for a maximizing sequence.

\begin{proposition}[Profile decomposition]\label{prop:profile-decomposition}
Let $P_n\in\Pcal_{\rm sym}((\R^\d)^N)$ have densities $\rho_n$ satisfying
\begin{equation}
  \int_{\R^\d}\rho_n=N,
  \qquad
  \sup_n\Lcal(\rho_n)<\infty.
\end{equation}
Assume also that the uniform pair-cost bound \eqref{eq:uniform-plan-pair-cost} holds.
After passage to a subsequence, there exist at most countably many truncated grand-canonical profiles $\Gamma^{(\alpha)}$, with densities $\rho^{(\alpha)}$ and mean particle numbers $m_\alpha$, such that
\begin{align}
  \sum_\alpha m_\alpha&\leq N,
  \label{eq:profile-mass-budget}\\
  \sum_\alpha\Lcal(\rho^{(\alpha)})
  &\leq\liminf_n\Lcal(\rho_n),
  \label{eq:profile-L-budget}\\
  \lim_nD(\rho_n)&=\sum_\alpha D(\rho^{(\alpha)}),
  \label{eq:profile-D-identity}\\
  \sum_\alpha\Wcal(\Gamma^{(\alpha)})
  &\leq\liminf_n
  \int_{(\R^\d)^N}\sum_{i<j}\frac1{|x_i-x_j|^\s}\diff P_n.
  \label{eq:profile-cost-ineq}
\end{align}
Consequently,
\begin{equation}\label{eq:profile-gain-ineq}
  \limsup_{n\to\infty}
  \left[
  \frac12D(\rho_n)
  -\int_{(\R^\d)^N}\sum_{i<j}\frac1{|x_i-x_j|^\s}\diff P_n
  \right]
  \leq\sum_\alpha\Gcal(\Gamma^{(\alpha)}).
\end{equation}
\end{proposition}

\begin{proof}
Combine \Cref{prop:density-profiles,prop:Hartree-profile,prop:cost-profile} and \eqref{eq:grand-canonical-profile-density}.  The first two inequalities follow from \Cref{prop:density-profiles}.  \Cref{eq:profile-gain-ineq} follows by subtracting \eqref{eq:profile-cost-ineq} from one half of \eqref{eq:profile-D-identity}.
\end{proof}

\begin{remark}[Absolute convergence of profile gains]\label{rem:profile-gain-absolute-convergence}
Under the hypotheses of \Cref{prop:profile-decomposition}, the series in \eqref{eq:profile-gain-ineq} is absolutely convergent, since
\begin{equation}
  \sum_\alpha|\Gcal(\Gamma^{(\alpha)})|
  \leq\frac12\sum_\alpha D(\rho^{(\alpha)})
  +\sum_\alpha\Wcal(\Gamma^{(\alpha)})<\infty.
\end{equation}
\end{remark}

\section{Completion of fluctuating profiles}\label{sec:grand-canonical}

Let $K$ denote the particle number of the profile, and suppose that its mean is below $N$ while $\Var(K)>0$.  To show that this profile cannot attain the $N$-particle constant, we complete each $k$-particle sector by adding $N-k$ particles and construct an exact $N$-particle trial state with a strictly larger quotient.  The added particles are distributed over two annuli.  Their spatial density places weight $\varepsilon$ on a large fixed annulus.  Particle-number fluctuations then produce a positive contribution of order $\varepsilon$: the Hartree cross term is larger than the averaged interaction between the original and added particles.  The inner radius is chosen so that this contribution dominates the order-$\varepsilon$ overlap with the original density.  The remaining weight is placed on a much more distant annulus, where its interaction and local-density contributions are negligible.  All remaining errors have order $\varepsilon^p$ or $\varepsilon^2$; since $p>1$, they are smaller than the positive order-$\varepsilon$ term when $\varepsilon$ is sufficiently small.

\subsection{Completion to the fixed particle number}

The operation below is related to canonicalization constructions for grand-canonical Coulomb and Riesz states, such as \cite[Lemma~3.2]{LewinLiebSeiringer2018}, but it is performed directly at the prescribed finite particle number and retains the exact cross-term identity needed for strictness.

Let $\Gamma=(\Gamma^{(k)})_{k=0}^N$ be a truncated grand-canonical probability with sectors at most $N$, as in \Cref{def:grand-canonical}, and let $\eta$ be a probability density on $\R^\d$.  For a finite measure $Q$ on $(\R^\d)^N$, write
\begin{equation}
  \Sym_N Q:=\frac1{N!}\sum_{\sigma\in\mathfrak S_N}\sigma_\#Q.
\end{equation}
With the convention $\Gamma^{(0)}\otimes\eta^{\otimes N}=\Gamma^{(0)}\eta^{\otimes N}$, define
\begin{equation}\label{eq:canonical-completion-plan}
  P_{\Gamma,\eta}
  :=\sum_{k=0}^N\Sym_N
  \bigl(\Gamma^{(k)}\otimes\eta^{\otimes(N-k)}\bigr).
\end{equation}
This is a canonical $N$-particle probability.  Identifying $P_{\Gamma,\eta}$ with the truncated grand-canonical probability $(0,\ldots,0,P_{\Gamma,\eta})$ concentrated in the $N$-particle sector, we may evaluate $\Gcal$ on $P_{\Gamma,\eta}$.  Let $q:=N-\E K$ be the mean number of particles added by the completion.
The one-body density of the completed probability is
\begin{equation}\label{eq:completed-density}
  \rho_{P_{\Gamma,\eta}}=\rho_\Gamma+q\eta.
\end{equation}
The Hartree energy of the completed density in \eqref{eq:completed-density} is $\frac12D(\rho_\Gamma+q\eta)$, whose cross term between the original density $\rho_\Gamma$ and the mean added density $q\eta$ is $qD(\rho_\Gamma,\eta)$.  In sector $k$, however, the completion adds exactly $N-k$ particles.  The following signed density records the discrepancy between this Hartree cross term and the averaged interaction between the original and added particles.
\begin{equation}\label{eq:centered-number-density}
  h_\Gamma
  :=q\rho_\Gamma-\sum_{k=0}^N(N-k)\rho_k
  =\sum_{k=0}^N(k-\E K)\rho_k.
\end{equation}

\begin{lemma}[Completion identity]\label{lem:completion-identity}
With the notation above, assume in addition that
\begin{equation}
  D(\eta)<\infty\quad\text{and}\quad\int_{\R^\d}|h_\Gamma(x)|U_\eta(x)\diff x<\infty.
\end{equation}
Then $\int_{\R^\d} h_\Gamma$ is finite, the signed pairing $D(h_\Gamma,\eta)$ is absolutely convergent, and every energy term in \eqref{eq:completion-identity} is finite.  Moreover,
\begin{equation}\label{eq:h-mass}
  \int_{\R^\d} h_\Gamma=\Var(K).
\end{equation}
The completion identity is
\begin{equation}\label{eq:completion-identity}
  \Gcal(P_{\Gamma,\eta})-\Gcal(\Gamma)
  =D(h_\Gamma,\eta)+\frac{q-\Var(K)}{2}D(\eta).
\end{equation}
\end{lemma}

\begin{proof}
Since $0\leq\rho_k\leq\rho_\Gamma$, we have $D(\rho_k)\leq D(\rho_\Gamma)<\infty$.  Positive definiteness of the Riesz form yields the energy Cauchy--Schwarz inequality $D(\rho_k,\eta)^2\leq D(\rho_k)D(\eta)<\infty$.  Thus, all unsigned cross terms in the expansion below are finite; the signed term is finite by hypothesis.
First, unpacking the definition of $h_\Gamma$ gives
\begin{equation}
  \int_{\R^\d} h_\Gamma=q\E K-\E[K(N-K)]=\E K^2-(\E K)^2=\Var(K).
\end{equation}
After expanding $\Gcal(P_{\Gamma,\eta})$ according to \eqref{eq:grand-canonical-gain} and canceling
$\Gcal(\Gamma)=\frac12D(\rho_\Gamma)-\Wcal(\Gamma)$, the remaining terms are
\begin{align}
 &qD(\rho_\Gamma,\eta)-\sum_{k=0}^N(N-k)D(\rho_k,\eta)
   +\left(\frac{q^2}{2}-\E\binom{N-K}{2}\right)D(\eta)\notag\\
 &=D\left(q\rho_\Gamma-\sum_{k=0}^N(N-k)\rho_k,\eta\right)
   +\frac12\left(q^2-
   \E\bigl[(N-K)(N-K-1)\bigr]\right)D(\eta)\notag\\
 &=D(h_\Gamma,\eta)
   +\frac12\left(q^2-\bigl(q^2+\Var(K)\bigr)+q\right)D(\eta)\notag\\
 &=D(h_\Gamma,\eta)+\frac{q-\Var(K)}{2}D(\eta),
\end{align}
where $\E[N-K]=q$ and $\E[(N-K)^2]=q^2+\Var(K)$.  This proves
\eqref{eq:completion-identity}.
\end{proof}

\subsection{Annular estimates}

To prepare the strict-completion argument below, we record two far-field estimates: the Riesz interaction with a distant annulus is governed to leading order by total signed mass, whereas the corresponding local-density overlap is of lower order.

Fix a smooth radial probability density $\nu$ supported in
\begin{equation}
  \{x\in\R^\d:1<|x|<2\},
\end{equation}
and define
\begin{equation}\label{eq:nuR}
  \nu_R(x):=R^{-\d}\nu(x/R).
\end{equation}

\begin{lemma}[Annular asymptotics]\label{lem:annulus-asymptotics}
Let $f\in L^1(\R^\d)$ be real-valued.  Then
\begin{equation}\label{eq:signed-annulus}
  R^\s D(f,\nu_R)\longrightarrow
  \left(\int_{\R^\d}\frac{\nu(y)}{|y|^\s}\diff y\right)\int_{\R^\d} f
  \qquad(R\to\infty).
\end{equation}
If $\rho\geq0$ belongs to $L^1(\R^\d)$, then
\begin{equation}\label{eq:overlap-annulus}
  R^\s\int_{\R^\d}\rho^{\s/\d}\nu_R\longrightarrow0.
\end{equation}
\end{lemma}

\begin{proof}
The potential of $\nu_R$ is
\begin{equation}
  U_{\nu_R}(x)=R^{-\s}U_\nu(x/R).
\end{equation}
Because $0<\s<\d$ and $\nu$ is smooth and compactly supported, $U_\nu$ is bounded and continuous, with $U_\nu(0)=\int_{\R^\d}\frac{\nu(y)}{|y|^\s}\diff y>0$.  In particular, $D(f,\nu_R)$ is absolutely convergent for every real-valued $f\in L^1$.  Dominated convergence gives \eqref{eq:signed-annulus}.

Since $\nu_R\leq CR^{-\d}$ and is supported in $\{R<|x|<2R\}$, H\"older's inequality gives
\begin{equation}
  R^\s\int_{\R^\d}\rho^{\s/\d}\nu_R
  \leq CR^{\s-\d}\int_{R<|x|<2R}\rho^{\s/\d}
  \leq C\left(\int_{R<|x|<2R}\rho\right)^{\s/\d}.
\end{equation}
The last factor tends to zero because $\rho\in L^1$.
\end{proof}

\subsection{Strict completion when the particle number fluctuates}

Combining the completion identity with the preceding annular estimates gives the strict comparison needed in the compactness argument: under the hypotheses below, a truncated grand-canonical probability with fluctuating particle number has quotient strictly below the fixed-$N$ value.

\begin{proposition}[Strict completion]\label{prop:strict-completion}
Let $\Gamma$ be a truncated grand-canonical probability with sectors at most $N$ such that
\begin{equation}\label{eq:fluctuating-profile-hypotheses}
  0<\E K<N,
  \qquad
  \Var(K)>0,
  \qquad
  0<\Lcal(\rho_\Gamma)<\infty.
\end{equation}
Then
\begin{equation}\label{eq:fluctuating-subcritical}
  \frac{\Gcal(\Gamma)}{\Lcal(\rho_\Gamma)}<\Lambda_N.
\end{equation}
\end{proposition}

\begin{proof}
To convert the given truncated grand-canonical probability $\Gamma$ into a canonical $N$-particle trial probability, we add the missing particles using a probability density supported on two spatial scales.  A fraction $\varepsilon$ of the missing particles is placed on a fixed large annulus, where it detects the positive total signed mass $\int_{\R^\d} h_\Gamma=\Var(K)$, which governs the leading interaction with the distant annulus.  The remaining fraction is placed on a much larger annulus so that its Riesz and local-density costs are negligible.  We first choose the inner radius, then choose $\varepsilon$, while the outer radius depends on $\varepsilon$.

Write $\rho=\rho_\Gamma$, $h=h_\Gamma$, and recall that $q=N-\E K$.  Set
\begin{equation}
  r:=\frac{\Gcal(\Gamma)}{\Lcal(\rho)}.
\end{equation}
If $r<0$ (recall that $\mathcal{G}$ is signed), the conclusion follows trivially from $\Lambda_N>0$.  Therefore, assume $r\geq0$.

By \eqref{eq:h-mass} and \Cref{lem:annulus-asymptotics},
\begin{equation}\label{eq:positive-h-shell}
  D(h,\nu_R)\sim\frac{\Var(K)}{R^\s}
  \int_{\R^\d}\frac{\nu(y)}{|y|^\s}\diff y>0.
\end{equation}
Also, with
\begin{equation}
  J_R:=\int_{\R^\d}\rho^{p-1}\nu_R
  =\int_{\R^\d}\rho^{\s/\d}\nu_R,
\end{equation}
we have $R^sJ_R\to0$.  Hence, we may fix $R$ sufficiently large that
\begin{equation}\label{eq:R-choice}
  pqr J_R<\frac12D(h,\nu_R).
\end{equation}

For $0<\varepsilon<1$, set
\begin{equation}\label{eq:two-scale-eta}
  S_\varepsilon:=R\varepsilon^{-2/\s},
  \qquad
  \eta_\varepsilon
  :=\varepsilon\nu_R+(1-\varepsilon)\nu_{S_\varepsilon}.
\end{equation}
Since $\nu_R$ and $\nu_{S_\varepsilon}$ are supported in the annuli $\{R<|x|<2R\}$ and $\{S_\varepsilon<|x|<2S_\varepsilon\}$, respectively, these annuli are disjoint for sufficiently small $\varepsilon$.  Set $P_\varepsilon:=P_{\Gamma,\eta_\varepsilon}$.  The density $\eta_\varepsilon$ is smooth and compactly supported, so $D(\eta_\varepsilon)<\infty$ and $U_{\eta_\varepsilon}\in L^\infty(\R^\d)$.  Since $h\in L^1(\R^\d)$, the hypotheses of \Cref{lem:completion-identity} are satisfied.  By \Cref{lem:completion-identity}, the change in the grand-canonical gain is
\begin{equation}\label{eq:deltaG-fluctuating}
  \delta\Gcal_\varepsilon
  :=\Gcal(P_\varepsilon)-\Gcal(\Gamma)
  =D(h,\eta_\varepsilon)+\frac{q-\Var(K)}{2}D(\eta_\varepsilon).
\end{equation}
Since $U_{\nu_{S_\varepsilon}}(x)=S_\varepsilon^{-\s}U_\nu(x/S_\varepsilon)$ and $U_\nu$ is bounded,
\begin{equation}\label{eq:h-outer-estimate}
  \left|D(h,\nu_{S_\varepsilon})\right|
  \leq\norm{h}_1\norm{U_\nu}_\infty S_\varepsilon^{-\s}
  =\norm{h}_1\norm{U_\nu}_\infty\frac{\varepsilon^2}{R^\s}.
\end{equation}
Moreover, since $D(\nu_T)=T^{-\s}D(\nu)$ and $D(\nu_R,\nu_{S_\varepsilon})=O(S_\varepsilon^{-\s})$, we have
\begin{equation}\label{eq:Deta-estimate}
  D(\eta_\varepsilon)
  =O\left(\frac{\varepsilon^2}{R^\s}\right).
\end{equation}
Combining \eqref{eq:deltaG-fluctuating}, \eqref{eq:h-outer-estimate}, and \eqref{eq:Deta-estimate}, we obtain
\begin{equation}\label{eq:deltaG-leading}
  \delta\Gcal_\varepsilon
  =\varepsilon D(h,\nu_R)
  +O\left(\frac{\varepsilon^2}{R^\s}\right).
\end{equation}

Let
\begin{equation}
  \delta\Lcal_\varepsilon
  :=\Lcal(\rho+q\eta_\varepsilon)-\Lcal(\rho).
\end{equation}
Since $1<p<2$, the elementary inequality
\begin{equation}\label{eq:power-expansion}
  (a+b)^p\leq a^p+p a^{p-1}b+C_p b^p,\qquad a,b\geq0,
\end{equation}
gives
\begin{equation}\label{eq:deltaL-bound0}
  \delta\Lcal_\varepsilon
  \leq pq\int_{\R^\d}\rho^{p-1}\eta_\varepsilon
  +C_pq^p\Lcal(\eta_\varepsilon).
\end{equation}

The outer-annulus estimate in \Cref{lem:annulus-asymptotics} gives
\begin{equation}
  \int_{\R^\d}\rho^{p-1}\nu_{S_\varepsilon}
  =O(S_\varepsilon^{-\s})
  =O(\varepsilon^2/R^\s),
\end{equation}
and disjointness of the two annuli gives
\begin{equation}
  \Lcal(\eta_\varepsilon)
  =\Lcal(\nu)
  \left(
  \frac{\varepsilon^p}{R^\s}
  +\frac{(1-\varepsilon)^p}{S_\varepsilon^\s}
  \right).
\end{equation}
Consequently,
\begin{equation}\label{eq:deltaL-leading}
  \delta\Lcal_\varepsilon
  \leq pq\varepsilon J_R
  +O\left(\frac{\varepsilon^p}{R^\s}\right)
  +O\left(\frac{\varepsilon^2}{R^\s}\right).
\end{equation}
Combining \eqref{eq:R-choice}, \eqref{eq:deltaG-leading}, and \eqref{eq:deltaL-leading},
\begin{equation}
  \delta\Gcal_\varepsilon-r\delta\Lcal_\varepsilon
  \geq\varepsilon\left[D(h,\nu_R)-pqrJ_R\right]
  -O\left(\frac{\varepsilon^p+\varepsilon^2}{R^\s}\right)>0
\end{equation}
for all sufficiently small $\varepsilon$, since $p>1$.  The preceding inequality and the definitions of $r$, $\delta\Gcal_\varepsilon$, and $\delta\Lcal_\varepsilon$ give the first inequality in the display below.  Since $P_\varepsilon$ is an $N$-particle trial probability with one-body density $\rho+q\eta_\varepsilon$, the definition of $\Wcal_N$ in \eqref{eq:WN-intro}, together with the definition of $\Gcal_N$ in \eqref{eq:GN-intro}, gives $\Gcal(P_\varepsilon)\leq\Gcal_N(\rho+q\eta_\varepsilon)$.  Moreover, $\rho+q\eta_\varepsilon\in\Acal_N$ (recall \eqref{eq:AN-intro}).  Consequently,
\begin{equation}
  r
  <\frac{\Gcal(P_\varepsilon)}{\Lcal(\rho+q\eta_\varepsilon)}
  \leq\frac{\Gcal_N(\rho+q\eta_\varepsilon)}{\Lcal(\rho+q\eta_\varepsilon)}
  \leq\Lambda_N.
\end{equation}
\end{proof}

\section{Compactness under a strict lower-particle inequality}\label{sec:attainment}

We now apply the profile decomposition to a normalized maximizing sequence.  A deterministic profile with $k<N$ particles is controlled by $\Lambda_k$, whereas a profile with non-deterministic particle number is strictly improved by \Cref{prop:strict-completion}.  The only remaining possibility is a full $N$-particle profile, which gives the optimizer.

\begin{proposition}[Attainment]\label{prop:conditional-attainment}
Let $N\geq1$.  Assume
\begin{equation}\label{eq:strict-lower-threshold}
  \Lambda_N>\max_{1\leq k<N}\Lambda_k,
\end{equation}
where the maximum over the empty set is omitted when $N=1$.  Then $\Lambda_N$ is attained.
\end{proposition}

\begin{proof}
Choose a maximizing sequence $\rho_n\in\Acal_N$.  By the scaling invariance \eqref{eq:scaling-all}, we may assume without loss of generality that
\begin{equation}\label{eq:max-seq-normalization}
  \Lcal(\rho_n)=1.
\end{equation}
Let $P_n$ be corresponding optimal plans given by \Cref{lem:OT-existence}.  Then
\begin{equation}\label{eq:max-seq-gain}
  \frac12D(\rho_n)
  -\int_{(\R^\d)^N}\sum_{i<j}\frac1{|x_i-x_j|^\s}\diff P_n
  \longrightarrow\Lambda_N.
\end{equation}
By \eqref{eq:D-mass-L} and \eqref{eq:max-seq-normalization}, the Hartree energies are uniformly bounded.  Hence, \eqref{eq:max-seq-gain} also shows that the pair costs of $P_n$ are uniformly bounded.  Apply \Cref{prop:profile-decomposition}.  The profile family cannot be empty, since then \eqref{eq:profile-gain-ineq} would contradict $\Lambda_N>0$.  For every nonzero profile, define
\begin{equation}\label{eq:profile-quotient}
  r_\alpha
  :=\frac{\Gcal(\Gamma^{(\alpha)})}
  {\Lcal(\rho^{(\alpha)})}.
\end{equation}
To obtain a contradiction, assume that every profile has mean particle number $m_\alpha<N$.
If the particle number of $\Gamma^{(\alpha)}$ fluctuates (i.e.,~$\Var(K)>0$), \Cref{prop:strict-completion} gives
\begin{equation}\label{eq:profile-fluctuating-gap}
  r_\alpha<\Lambda_N.
\end{equation}
If it is deterministic, then it equals an integer $k<N$.  A nonzero profile cannot have $k=0$; hence $1\leq k<N$, the profile is a genuine canonical $k$-particle probability, and
\begin{equation}\label{eq:profile-deterministic-gap}
  r_\alpha\leq\Lambda_k<\Lambda_N.
\end{equation}
The second inequality follows from the hypothesis \eqref{eq:strict-lower-threshold}.
Thus, every profile has a pointwise strict gap.

There can be only finitely many profiles whose mean particle number $m_\alpha$ exceeds a fixed cutoff.  Those finitely many have a positive minimum gap from $\Lambda_N$, whereas every remaining small-mass profile has a uniformly small quotient by the Hardy--Littlewood--Sobolev estimate.  We now use this dichotomy to upgrade the pointwise bounds $r_\alpha<\Lambda_N$ to the uniform estimate $\sup_\alpha r_\alpha<\Lambda_N$.  Indeed, \eqref{eq:D-mass-L} and nonnegativity of the pair cost give
\begin{equation}\label{eq:small-profile-bound}
  r_\alpha
  \leq\frac{\frac12D(\rho^{(\alpha)})}
  {\Lcal(\rho^{(\alpha)})}
  \leq C_{\s,\d} m_\alpha^{1-\s/\d}.
\end{equation}
Choose $\delta>0$ so small that $C_{\s,\d}\delta^{1-\s/\d}<\Lambda_N/2$.  All profiles with $m_\alpha<\delta$ have quotient at most $\Lambda_N/2$.  Since $\sum_\alpha m_\alpha\leq N$, only finitely many profiles have mass at least $\delta$, and each of these finitely many has quotient strictly below $\Lambda_N$.  Therefore,
\begin{equation}\label{eq:uniform-profile-gap}
  r_*:=\sup_\alpha r_\alpha<\Lambda_N.
\end{equation}
Moreover, $r_*>0$.  Indeed, the profile gain inequality \eqref{eq:profile-gain-ineq} and $\Lambda_N>0$ give
\begin{equation}
  0<\Lambda_N
  \leq\sum_\alpha\Gcal(\Gamma^{(\alpha)}),
\end{equation}
so some profile has positive quotient.  Therefore, the inequality \eqref{eq:profile-L-budget}, using \eqref{eq:max-seq-normalization}, implies
\begin{equation}
  r_*\sum_\alpha\Lcal(\rho^{(\alpha)})\leq r_*.
\end{equation}
Combining the preceding two displays with \eqref{eq:profile-quotient} and \eqref{eq:uniform-profile-gap}, we obtain the contradiction
\begin{equation}
  \Lambda_N\leq\sum_\alpha\Gcal(\Gamma^{(\alpha)})
  =\sum_\alpha r_\alpha\Lcal(\rho^{(\alpha)})
  \leq r_*\sum_\alpha\Lcal(\rho^{(\alpha)})
  \leq r_*<\Lambda_N.
\end{equation}

Hence, some profile has $m_\alpha=N$.  Since its particle number satisfies $0\le K_\alpha \le N$, it follows that this profile is a canonical $N$-particle probability.  The inequality \eqref{eq:profile-mass-budget} implies this is the unique nonzero profile.  Then \eqref{eq:profile-gain-ineq} and \eqref{eq:profile-L-budget} imply
\begin{equation}
  \Lambda_N
  \leq\Gcal(\Gamma^{(\alpha)})
  \leq\Lambda_N\Lcal(\rho^{(\alpha)})
  \leq\Lambda_N.
\end{equation}
Since $\Lambda_N>0$, equality holds throughout.  In particular,
\begin{equation}
  \Lcal(\rho^{(\alpha)})=1,
  \qquad
  \Gcal(\Gamma^{(\alpha)})=\Lambda_N,
\end{equation}
so $\rho^{(\alpha)}$ is an optimizer.
\end{proof}

\section{Strict increase in the particle number}\label{sec:one-particle-extension}

The preceding compactness argument produces an optimizer once the fixed-$N$ value $\Lambda_N$ is strictly larger than all lower-particle values $\Lambda_k$ for $k<N$.  To continue the induction, one must create that separation at the next particle number.  Compact support is used only here.  If $P$ is an optimal $N$-particle plan associated with a compactly supported optimizing density $\rho$, then $X=(x_1,\ldots,x_N)\sim P$ belongs almost surely to $K^N$ for some compact set $K$.  This leaves an exterior region in which the interaction between $X$ and an added particle can be varied without paying a first-order local-density cost.

\begin{proposition}[Di Marino--Lelotte \cite{DiMarinoLelotte2026}]\label{prop:DML}
Let $\d\geq1$ and $0<\s<\d$.  If the supremum defining $\Lambda_N$ is attained, then every optimizing density is compactly supported.
\end{proposition}

In the notation of \cite{DiMarinoLelotte2026}, their $C(\rho)$ is our $\Wcal_N(\rho)$ and their sharp fixed-$N$ constant $c_{\mathrm{LO}}(\s,\d,N)$ equals $\Lambda_N(\s,\d)$.  Thus, a minimiser of the sharp fixed-$N$ Lieb--Oxford bound in their terminology is precisely an optimizing density in ours.

The construction below separates two roles of the added particle.  A small fraction of its one-body density is placed in a fixed compact region outside $K$, and the corresponding added-particle position is coupled non-independently with $X$.  This produces a first-order reduction of the actual cross-interaction relative to the Hartree cross term.  The remaining mass is spread on a much larger scale, where its Riesz and $L^p$ costs are negligible.  Thus, an order-$\varepsilon$ interaction gain dominates an order-$\varepsilon^p$ local-density cost.

\subsection{Exterior uniqueness for finite atomic measures}

The non-product coupling requires that distinct finite atomic measures have distinct exterior Riesz potentials.  For $\d\geq2$, this is a real-analytic continuation statement.  In dimension one, where removing finitely many points disconnects the line, the corresponding statement follows from the exterior moment expansion.

\begin{lemma}[Exterior uniqueness]\label{lem:exterior-injectivity}
Let
\begin{equation}
  \mu=\sum_{j=1}^M a_j\delta_{z_j}
\end{equation}
be a finite signed atomic measure on $\R^\d$, where the $z_j$ are distinct.  If
\begin{equation}\label{eq:atomic-riesz-potential}
  U_\mu(y):=\sum_{j=1}^M\frac{a_j}{|y-z_j|^\s}
\end{equation}
vanishes on a nonempty open subset of an unbounded connected component of $\R^\d\setminus\supp\mu$, then $\mu=0$.
\end{lemma}

\begin{proof}
Assume first that $\d\geq2$.  Each summand is real analytic away from its pole, being the composition of the polynomial $y\mapsto |y-z_j|^2$ with the real-analytic function $t\mapsto t^{-\s/2}$ on $(0,\infty)$; see \cite[Sections~1.4 and~2.2]{KrantzParks2002}.  Hence, $U_\mu$ is real analytic on
\begin{equation}
  \Omega:=\R^\d\setminus\{z_1,\ldots,z_M\}.
\end{equation}
Since $\Omega$ is open and connected and $U_\mu$ is real analytic on $\Omega$, \cite[Proposition~1]{Mityagin2020} implies that either $U_\mu$ vanishes identically or its zero set has Lebesgue measure zero.  The latter alternative is impossible because $U_\mu$ vanishes on a nonempty open subset of $\Omega$, which has positive Lebesgue measure.  Hence, $U_\mu$ vanishes throughout $\Omega$.  For every $j$,
\begin{equation}
  a_j=\lim_{y\to z_j}|y-z_j|^sU_\mu(y)=0,
\end{equation}
because all terms with index different from $j$ remain bounded near $z_j$.

Now let $\d=1$.  After reflection if necessary, the open set lies in the component to the right of the atoms.  Real-analytic continuation on that component shows that $U_\mu(y)=0$ for every sufficiently large $y$.  Choose $R>\max_j|z_j|$.  For $y>R$, the binomial series converges absolutely and gives
\begin{equation}\label{eq:moment-expansion}
  0=y^sU_\mu(y)=\sum_{j=1}^M a_j\left(1-\frac{z_j}{y}\right)^{-\s}=\sum_{n=0}^\infty\frac{(\s)_n}{n!}\,y^{-n}\sum_{j=1}^M a_jz_j^n,
\end{equation}
where $(\s)_n=\s(\s+1)\cdots(\s+n-1)$ and $(\s)_0=1$.  The power series in $y^{-1}$ vanishes on an interval, hence all of its coefficients vanish:
\begin{equation}
  \sum_{j=1}^M a_jz_j^n=0
  \qquad(n\geq0).
\end{equation}
Taking $n=0,\ldots,M-1$ and using the invertibility of the Vandermonde matrix associated with the distinct points $z_j$ gives $a_1=\cdots=a_M=0$.
\end{proof}

\subsection{A non-product coupling with lower cross-interaction}

Let $\rho\in\Acal_N$ be compactly supported and let $P$ be an optimal plan.  Choose a compact set $K\subset\R^\d$ with $\supp\rho\subset K$.  View $X=(x_1,\ldots,x_N)$ as the canonical coordinate random vector with law $P$.  Since $\rho_P=\rho$,
\begin{equation}
  0=\rho(K^c)=\sum_{i=1}^N P(x_i\notin K),
\end{equation}
we have $P(K^N)=1$.  For $X\in K^N$ and $y\notin K$, define
\begin{equation}\label{eq:WXY}
  W(X,y):=\sum_{i=1}^N\frac1{|x_i-y|^\s}.
\end{equation}

For an exterior probability density $\nu$ supported in $K^c$, consider the auxiliary two-marginal optimal-transport problem
\begin{equation}
  \inf\left\{
    \int_{K^N\times K^c}W(X,y)\diff\pi(X,y):
    \text{$\pi$ has marginals $P$ and $\nu$}
  \right\}.
\end{equation}
This problem is distinct from the multimarginal problem defining $\Wcal_N(\rho)$: the law $P$ of the original $N$-particle configuration $X$ is fixed, and only the joint law of $X$ and the added particle $y$ is varied.  The product coupling $P\otimes\nu$ makes $X$ and $y$ independent, and its cost is
\begin{equation}
  \int_{K^N\times K^c}W\diff(P\otimes\nu)=D(\rho,\nu).
\end{equation}

If $P\otimes\nu$ were optimal, its support $\supp P\times\supp\nu$ would be $W$-cyclically monotone.\footnote{For a minimization cost $c$, a set $\Gamma$ is $c$-cyclically monotone if, for every $m\geq2$, every $(X_j,y_j)\in\Gamma$, $1\leq j\leq m$, and every permutation $\sigma$ of $\{1,\ldots,m\}$, one has $\sum_{j=1}^m c(X_j,y_j)\leq\sum_{j=1}^m c(X_j,y_{\sigma(j)})$.  The case $m=2$ is the two-cycle condition; see \cite{SmithKnott1992,GangboMcCann1996,Ruschendorf1996}.}  Hence, for any two configurations $X,X'\in\supp P$ and any two exterior points $y,y'\in\supp\nu$, applying the two-cycle condition first to $(X,y),(X',y')$ and then to $(X,y'),(X',y)$ gives
\begin{equation}
  W(X,y)+W(X',y')=W(X,y')+W(X',y).
\end{equation}
Equivalently, $W$ would have the form $a(X)+b(y)$ on this product support.  The proof below chooses $\nu$ so that this degeneracy fails and constructs a cheaper coupling explicitly.

Namely, it perturbs the product coupling itself by setting
\begin{equation}
  \diff\pi_t(X,y)
  =\bigl(1-tf(X)g(y)\bigr)\diff P(X)\nu(y)\diff y,
\end{equation}
where $f$ and $g$ are bounded and have zero mean with respect to $P$ and $\nu$, respectively.  The zero-mean conditions preserve both marginals, while the choice made below lowers the cross-interaction strictly, yielding
\begin{equation}
  \int_{K^N\times K^c}W\diff\pi_t=D(\rho,\nu)-\delta
  \qquad\text{for some }\delta>0.
\end{equation}
This is the quantitative defect used in the subsequent Lieb--Oxford quotient comparison.

\begin{lemma}[Lower cross-interaction]\label{lem:strict-coupling}
Assume $\rho\in L^p$ is nonzero.  There exist
\begin{itemize}
\item a smooth compactly supported probability density $\nu$ on $\R^\d$ with $\supp\nu\cap K=\varnothing$,
\item a coupling $\pi$ of $P$ and $\nu$,
\item a number $\delta>0$,
\end{itemize}
such that
\begin{equation}\label{eq:strict-coupling}
  \int_{(\R^\d)^N\times\R^\d} W(X,y)\diff\pi(X,y)
  =D(\rho,\nu)-\delta.
\end{equation}
\end{lemma}

\begin{proof}
It suffices to find an exterior observable whose interaction with the random $N$-particle configuration is nonconstant, and then perturb the product coupling in a bounded zero-marginal direction.

\smallskip\noindent\emph{Step 1: the exterior interaction is not deterministic.}
For $X\in K^N$, let
\begin{equation}
  \mu_X:=\sum_{i=1}^N\delta_{x_i}
\end{equation}
be its counting measure, which has total mass $N$.  The random measure $\mu_X$ cannot be $P$-almost surely constant, because otherwise
\begin{equation}
  \rho=\int_{(\R^\d)^N} \mu_X\diff P(X)
\end{equation}
would be a finite atomic measure, contradicting $\rho\in L^p$.

Let $U$ be an unbounded connected component of $\R^\d\setminus K$.  We claim that there exist $y_0,y_1\in U$ for which
\begin{equation}\label{eq:H-nonconstant}
  H(X):=W(X,y_0)-W(X,y_1)
\end{equation}
is not $P$-almost surely constant.  Otherwise, we may fix $y_*\in U$ and a countable dense subset $\mathcal D\subset U$ such that, for every $y\in\mathcal D$, there exist $c_y\in\R$ and a $P$-full set $\mathcal X_y\subset K^N$ satisfying
\begin{equation}
  W(X,y_*)-W(X,y)=c_y\qquad(X\in\mathcal X_y).
\end{equation}
Since $\mathcal D$ is countable, the set
\begin{equation}
  \mathcal X_0:=\bigcap_{y\in\mathcal D}\mathcal X_y
\end{equation}
is $P$-full.  Hence, for $X,X'\in\mathcal X_0$ and $y\in\mathcal D$, applying the preceding identity to $X$ and $X'$ gives
\begin{equation}
  W(X,y)-W(X',y)=W(X,y_*)-W(X',y_*).
\end{equation}
Thus, the left-hand side is independent of $y\in\mathcal D$ and, by continuity, of $y\in U$.  Since $U$ is unbounded and both $W(X,y)$ and $W(X',y)$ tend to zero as $|y|\to\infty$ inside $U$, this difference vanishes throughout $U$.  Applying \Cref{lem:exterior-injectivity} to the finite signed atomic measure $\mu_X-\mu_{X'}$ gives $\mu_X=\mu_{X'}$.  Thus, $\mu_X$ would be almost surely constant, a contradiction.

\smallskip\noindent\emph{Step 2: a zero-marginal perturbation lowers the cost.}
Replace the point masses at $y_0,y_1$ by smooth probability densities $\nu_0,\nu_1$ supported in sufficiently small disjoint balls contained in $U$, and set
\begin{equation}
  \nu:=\frac12(\nu_0+\nu_1).
\end{equation}
Let $g=1$ on $\supp\nu_0$ and $g=-1$ on $\supp\nu_1$.  Then $\int_{\R^\d} g\diff\nu=0$.  Define the bounded function
\begin{equation}
  \overline H(X):=\int_{\R^\d} g(y)W(X,y)\nu(y)\diff y
\end{equation}
Since $H$ is nonconstant, $\operatorname{dist}_{L^2(P)}(H,\mathbb R)=\sqrt{\Var_P(H)}>0$.  By uniform continuity of $W$ on $K^N$ times a compact exterior neighborhood of $\{y_0,y_1\}$, the balls may be chosen so small that $\norm{\overline H-\frac12H}_{L^\infty(P)}<\frac14\sqrt{\Var_P(H)}$.  It follows that $\operatorname{dist}_{L^2(P)}(\overline H,\mathbb R)\geq\frac14\sqrt{\Var_P(H)}>0$, so $\overline H$ remains nonconstant in $L^2(P)$.  Set
\begin{equation}
  f(X):=\overline H(X)-\int_{(\R^\d)^N}\overline H\diff P.
\end{equation}
Then $\int_{(\R^\d)^N} f\diff P=0$ and
\begin{equation}\label{eq:variance-covariance}
  \iint_{(\R^\d)^N\times\R^\d} f(X)g(y)W(X,y)\diff P(X)\nu(y)\diff y
  =\Var_P(\overline H)>0.
\end{equation}
For sufficiently small $t>0$, define
\begin{equation}
  \diff\pi(X,y)
  :=\bigl(1-tf(X)g(y)\bigr)\diff P(X)\nu(y)\diff y.
\end{equation}
The density is nonnegative.  The two zero-mean identities show that the marginals of $\pi$ are $P$ and $\nu$.  Its cost is
\begin{equation}
  \int_{(\R^\d)^N\times\R^\d} W\diff\pi
  =D(\rho,\nu)-t\Var_P(\overline H).
\end{equation}
Taking $\delta=t\Var_P(\overline H)$ yields the desired conclusion.
\end{proof}

\subsection{Strict extension by one particle}

In the $N$-to-$(N+1)$ case of the diffuse-particle construction in \Cref{prop:monotonicity}, the added particle is independent of the original $N$-particle configuration.  The proof below replaces an $\varepsilon$-weighted part of that product coupling by the non-product coupling from \Cref{lem:strict-coupling}.

\begin{proposition}[Strict extension]\label{prop:strict-extension}
If $\Lambda_N$ is attained, then
\begin{equation}\label{eq:strict-next}
  \Lambda_{N+1}>\Lambda_N.
\end{equation}
\end{proposition}

\begin{proof}
Let $\rho$ attain $\Lambda_N$, and choose an optimal plan $P$.  By \Cref{prop:DML}, $\rho$ is compactly supported.  Apply \Cref{lem:strict-coupling} to obtain $\nu$, $\pi$, and $\delta>0$.

Choose a smooth compactly supported probability density $\zeta$ on $\R^\d$, and choose $A>0$ so that $\supp\zeta\subset B_A$.  Fix a unit vector $e$ and, for $R\gg1$, set
\begin{equation}
  b_R:=3AR e,
  \qquad
  \zeta_R(x):=R^{-\d}\zeta\left(\frac{x-b_R}{R}\right).
\end{equation}
The support of $\zeta_R$ lies in $B_{AR}(b_R)$ and is therefore disjoint from both $\supp\rho$ and $\supp\nu$ for all sufficiently large $R$.  By translation and scaling invariance,
\begin{equation}\label{eq:zeta-scaling}
  D(\zeta_R)=R^{-\s}D(\zeta),
  \qquad
  \Lcal(\zeta_R)=R^{-\s}\Lcal(\zeta).
\end{equation}
For $0<\varepsilon<1$, choose $R_\varepsilon=\varepsilon^{-2/\s}$ and set
\begin{equation}\label{eq:added-density}
  \eta_\varepsilon
  :=\varepsilon\nu+(1-\varepsilon)\zeta_{R_\varepsilon}.
\end{equation}
Identifying $(\R^\d)^N\times\R^\d$ with $(\R^\d)^{N+1}$, define the symmetric probability measure
\begin{equation}\label{eq:distinguished-plan}
  Q_\varepsilon
  :=\Sym_{N+1}\!\left[
  \varepsilon\pi
  +(1-\varepsilon)(P\otimes\zeta_{R_\varepsilon})
  \right].
\end{equation}
Symmetrization preserves the sum of the one-body marginals, so
\begin{equation}\label{eq:Q-density}
  \rho_{Q_\varepsilon}=\rho+\eta_\varepsilon.
\end{equation}
Before symmetrization, both terms in the convex combination in \eqref{eq:distinguished-plan} have $P$ as their marginal on the first $N$ coordinates.  Decomposing the pair cost into the interactions within those coordinates and the interactions with the last coordinate, and using the optimality of $P$, \eqref{eq:strict-coupling}, and invariance of the total pair cost under symmetrization, we obtain
\begin{equation}\label{eq:Q-pair-cost}
  \int_{(\R^\d)^{N+1}}\sum_{1\leq i<j\leq N+1}\frac1{|x_i-x_j|^\s}\diff Q_\varepsilon
  =\Wcal_N(\rho)
  +\varepsilon\bigl(D(\rho,\nu)-\delta\bigr)
  +(1-\varepsilon)D(\rho,\zeta_{R_\varepsilon}).
\end{equation}
Since
\begin{equation}
  D(\rho,\eta_\varepsilon)
  =\varepsilon D(\rho,\nu)
  +(1-\varepsilon)D(\rho,\zeta_{R_\varepsilon}),
\end{equation}
we obtain the identity
\begin{equation}\label{eq:strict-extension-identity}
  \frac12D(\rho+\eta_\varepsilon)
  -\int_{(\R^\d)^{N+1}}\sum_{1\leq i<j\leq N+1}\frac1{|x_i-x_j|^\s}\diff Q_\varepsilon
  =\Gcal_N(\rho)+\varepsilon\delta+\frac12D(\eta_\varepsilon).
\end{equation}
Since $Q_\varepsilon$ is admissible for $\rho+\eta_\varepsilon$ and $D(\eta_\varepsilon)\geq0$, it follows that
\begin{equation}\label{eq:strict-extension-gain}
  \Gcal_{N+1}(\rho+\eta_\varepsilon)
  \geq\Gcal_N(\rho)+\varepsilon\delta.
\end{equation}

Since $R_\varepsilon\to\infty$ as $\varepsilon\downarrow0$, the supports of $\rho$, $\nu$, and $\zeta_{R_\varepsilon}$ are pairwise disjoint for all sufficiently small $\varepsilon>0$.  Hence, by \eqref{eq:added-density}, our choice $R_\varepsilon=\varepsilon^{-2/\s}$, and recalling that $p=1+\frac\s\d$,
\begin{equation}\label{eq:strict-extension-L}
  \Lcal(\rho+\eta_\varepsilon)=\Lcal(\rho)+\varepsilon^p\Lcal(\nu)+(1-\varepsilon)^p\Lcal(\zeta_{R_\varepsilon})=\Lcal(\rho)+O(\varepsilon^p)+O(\varepsilon^2).
\end{equation}
Since $\Gcal_N(\rho)=\Lambda_N\Lcal(\rho)$, \eqref{eq:strict-extension-gain} and \eqref{eq:strict-extension-L} imply
\begin{equation}
  \frac{\Gcal_{N+1}(\rho+\eta_\varepsilon)}
  {\Lcal(\rho+\eta_\varepsilon)}
  >\Lambda_N
\end{equation}
for all sufficiently small $\varepsilon$.  Taking the supremum over $\mathcal{A}_{N+1}$ proves \eqref{eq:strict-next}.
\end{proof}

\section{Proof of the main theorem and consequences}\label{sec:conclusion}

\subsection{Proof of the main theorem}

\begin{proof}[Proof of \Cref{thm:main}]
For $N=1$, the strict lower-particle inequality in \Cref{prop:conditional-attainment} is vacuous, and hence $\Lambda_1$ is attained.

As our induction hypothesis, assume that $\Lambda_N$ is attained.  By \Cref{prop:strict-extension},
\begin{equation}
  \Lambda_{N+1}>\Lambda_N.
\end{equation}
The Lieb--Oxford monotonicity in \Cref{prop:monotonicity} gives
\begin{equation}
  \Lambda_k\leq\Lambda_N<\Lambda_{N+1}
  \qquad(1\leq k\leq N).
\end{equation}
Thus, the strict lower-particle inequality holds at particle number $N+1$, and \Cref{prop:conditional-attainment} implies that $\Lambda_{N+1}$ is attained.  Induction proves attainment for every finite $N$ and all the strict inequalities in \eqref{eq:strict-sequence}.  Compact support follows from \Cref{prop:DML}.
\end{proof}

\subsection{Consequences and outlook}

Since $\cLO(\s,\d)=\sup_{N\geq1}\Lambda_N(\s,\d)$ and the sequence $N\mapsto\Lambda_N(\s,\d)$ is strictly increasing, the universal Riesz Lieb--Oxford constant is not realized at any finite particle number.  \Cref{thm:main} also yields at least one compactly supported optimizer at every finite particle number.  However, this existence result does not characterize the optimizers or quantify their behavior.

The present argument gives no uniqueness statement for the optimizing density, even modulo the natural translation, rotation, and dilation symmetries, and no information about its symmetry, support geometry, or regularity beyond membership in $L^1\cap L^p$ and compact support.  It also does not address the structure or uniqueness of an optimal multimarginal plan at an optimizing density, which is a separate question.  An Euler--Lagrange or dual characterization of optimizing densities may provide a route to these questions, but the compactness-and-comparison proof developed here produces no such condition.

The strict gap between successive fixed-particle-number constants is likewise qualitative.  Once attainment at particle number $N$ has been established, the extension step starts from a particular compactly supported optimizing density and an associated optimal $N$-particle plan.  The resulting decrease in the cross-interaction and the range of occupation parameters for which it dominates are not controlled uniformly over such choices.  Making this argument quantitative would require estimates uniform in $N$ and over a normalized family of optimizers, after fixing the translation and dilation freedoms.  Such estimates could also inform the large-$N$ behavior of optimizers, including whether one can choose a normalization and topology in which finite-$N$ optimizers have a nontrivial limit and whether any such limit is related to the infinite-volume uniform-electron-gas problem.

The proof also does not extend directly to the endpoint regimes.  As $\s\downarrow0$, the logarithmic kernel appears after the renormalization $(|x-y|^{-\s}-1)/\s\to-\log|x-y|$.  At the same time, the exponent $p=1+\s/\d$ tends to one.  For $\s>0$, the strict-completion construction produces an interaction gain of order $\varepsilon$, while the local-density cost of the added density is of order $\varepsilon^p=o(\varepsilon)$.  At $\s=0$, both contributions are of order $\varepsilon$, so this construction no longer forces a strict gain.  For the hypersingular Riesz exponents $\s\geq\d$, the kernel is no longer locally integrable, the Hartree term need not be finite on the present admissible class, and the Hardy--Littlewood--Sobolev estimates used here are unavailable.  Moreover, $p\geq2$.  The power-remainder estimate \eqref{eq:power-expansion} is valid in the form used here for $1<p\leq2$; for $p>2$, an additional mixed remainder of order $a^{p-2}b^2$ is unavoidable.  At the borderline $\s=\d$ (so $p=2$), the algebraic estimate still holds, but the preceding local-integrability and Hartree-energy obstructions already prevent the present argument.  Analogous existence or strict-monotonicity results in either regime would therefore require a reformulated variational problem and new compactness and comparison mechanisms.

\appendix
\crefalias{section}{appendix}
\section{Density profile extraction and Riesz--Hartree decoupling}\label{app:density-profiles}

For completeness, we give the translation-profile construction used in \Cref{sec:profiles}.  The argument is an $L^1\cap L^p$ version of the standard concentration--compactness/profile-decomposition iteration (cf. \cite{Lions1984I,Solimini1995,Gerard1998,Jaffard1999}).  Its role in the body of the paper is to provide countably many separated density profiles, a vanishing remainder, and an exact identity for the Riesz--Hartree energy.

\subsection{Continuity and vanishing for the Riesz--Hartree energy}

Set
\begin{equation}\label{eq:diagonal-exponents}
  r_0:=\frac{\d+\s}{2\s}>1,
  \qquad
  \gamma:=\frac{\s(\d-\s)}{\d+\s}>0.
\end{equation}
The choice of $r_0$ is characterized by
\begin{equation}
  1+\frac1{p'}=\frac1{r_0}+\frac1p.
\end{equation}

\begin{lemma}[Diagonal control]\label{lem:near-diagonal}
If $f\in L^p(\R^\d)$ is nonnegative, then
\begin{equation}\label{eq:near-diagonal}
  \iint_{|x-y|<\delta}
  \frac{f(x)f(y)}{|x-y|^\s}\diff x\diff y
  \leq C_{\s,\d}\delta^\gamma\norm{f}_p^2.
\end{equation}
\end{lemma}

\begin{proof}
Let
\begin{equation}
  K_\delta(x):=|x|^{-\s}\one_{|x|<\delta}.
\end{equation}
Since $sr_0=(\d+\s)/2<\d$,
\begin{equation}
  \norm{K_\delta}_{r_0}
  =C_{\s,\d}\delta^{\d/r_0-\s}
  =C_{\s,\d}\delta^\gamma.
\end{equation}
Young's inequality and the defining relation for $r_0$ give
\begin{equation}
  \norm{K_\delta*f}_{p'}
  \leq\norm{K_\delta}_{r_0}\norm{f}_p.
\end{equation}
Pairing with $f\in L^p$ proves \eqref{eq:near-diagonal}.
\end{proof}

\begin{lemma}[Tight continuity]\label{lem:Hartree-tight}
Suppose $f_n\geq0$ is bounded in $L^1\cap L^p$, the measures $f_n\diff x$ are tight, and
\begin{equation}
  f_n\weak f\quad\text{in }L^p.
\end{equation}
Then
\begin{equation}\label{eq:Hartree-tight}
  D(f_n)\to D(f).
\end{equation}
\end{lemma}

\begin{proof}
Weak $L^p$ convergence and tightness imply narrow convergence
\begin{equation}
  f_n\diff x\weak f\diff x,
\end{equation}
and hence narrow convergence of the product measures.

Choose $\theta_R\in C_c(\R^\d)$ with $0\leq\theta_R\leq1$ and $\theta_R=1$ on $B_R$, and define
\begin{equation}
  k_\delta(z):=\min\{\delta^{-\s},|z|^{-\s}\},
  \qquad
  K_{\delta,R}(x,y)
  :=\theta_R(x)\theta_R(y)k_\delta(x-y).
\end{equation}
The kernel $K_{\delta,R}$ is bounded and continuous, so its double integral converges under the product measures.

It remains to control the two truncations uniformly.  Let
\begin{equation}
  \vartheta:=\frac{\d-\s}{2\d}>0.
\end{equation}
By the mixed Hardy--Littlewood--Sobolev estimate and interpolation,
\begin{equation}\label{eq:tight-tail-control}
  D(g,h)
  \leq C_{\s,\d}\norm{g}_1^\vartheta
  \norm{g}_p^{1-\vartheta}\norm{h}_\sigma
\end{equation}
for nonnegative $g,h\in L^1\cap L^p$.  Applying this with $g=f_n\one_{B_R^c}$ and $h=f_n$ shows that the terms removed by $1-\theta_R(x)\theta_R(y)$ tend to zero uniformly in $n$ as $R\to\infty$.  The same holds for $f$.  The difference between $|x-y|^{-\s}$ and $k_\delta(x-y)$ is supported on $|x-y|<\delta$ and is uniformly bounded by \Cref{lem:near-diagonal}.  First, pass to the limit in $n$ for fixed $R,\delta$, and then let $\delta\downarrow0$ and $R\to\infty$.
\end{proof}

\begin{lemma}[Vanishing]\label{lem:Hartree-vanishing}
Let $f_n\geq0$ be bounded in $L^1\cap L^p$.  If for every fixed $R>0$,
\begin{equation}\label{eq:vanishing-condition}
  \sup_{a\in\R^\d}\int_{B_R(a)}f_n\longrightarrow0,
\end{equation}
then
\begin{equation}\label{eq:Hartree-vanishing}
  D(f_n)\longrightarrow0.
\end{equation}
\end{lemma}

\begin{proof}
Fix $R>0$.  For $k\in\mathbb Z^\d$, let $Q_k:=R\bigl(k+[0,1)^\d\bigr)$ and $Q_k^*:=R\bigl(k+[-1,2)^\d\bigr)$.  Thus, $Q_k^*$ is the concentric enlargement of $Q_k$ from side length $R$ to side length $3R$.  If $x\in Q_k$ and $|x-y|<R$, then $x,y\in Q_k^*$.  Moreover, every point belongs to at most $3^\d$ cubes $Q_k^*$, and each $Q_k^*$ lies in a ball of radius $\frac32\sqrt{\d}\,R$.  By \eqref{eq:D-mass-L} and the overlap bound,
\begin{equation}\label{eq:vanishing-near}
  \iint_{|x-y|<R}\frac{f_n(x)f_n(y)}{|x-y|^\s}\diff x\diff y
  \leq C_{\s,\d}
  \left(\sup_a\int_{B_{\frac32\sqrt{\d}\,R}(a)}f_n\right)^{1-\s/\d}
  \Lcal(f_n).
\end{equation}
The contribution from $|x-y|\geq R$ is at most $R^{-\s}\norm{f_n}_1^2$.  For each fixed $R$, the near-field bound in \eqref{eq:vanishing-near} tends to zero as $n\to\infty$ by \eqref{eq:vanishing-condition}.  The uniform $L^1$ bound makes the far-field contribution arbitrarily small, uniformly in $n$, once $R$ is large.  Together, these estimates prove \eqref{eq:Hartree-vanishing}.
\end{proof}

\subsection{Density profiles under translations}

For a bounded sequence $f_n\geq0$ in $L^1$, set
\begin{equation}\label{eq:concentration-parameter}
  q\bigl[(f_n)_n\bigr]
  :=\lim_{R\to\infty}
  \left(\limsup_{n\to\infty}
  \sup_{a\in\R^\d}\int_{B_R(a)}f_n\right).
\end{equation}
This is the large-radius limiting value of the associated L\'evy concentration functions.

\begin{proposition}[Translation profiles]\label{prop:density-profiles}
Let $\rho_n\geq0$ satisfy
\begin{equation}\label{eq:density-bounds}
  \sup_n\left(\norm{\rho_n}_1+\Lcal(\rho_n)\right)<\infty.
\end{equation}
After passage to a subsequence, there exist at most countably many nonzero profiles $\rho^{(\alpha)}\in L^1\cap L^p$, with associated centers $a_n^{(\alpha)}\in\R^\d$, and, for every fixed $J$, a decomposition
\begin{equation}\label{eq:finite-decomposition}
  \rho_n
  =\sum_{\alpha=1}^J u_n^{(\alpha)}
  +r_n^{(J)}+e_n^{(J)},
\end{equation}
with nonnegative terms, such that:
\begin{enumerate}[label=\textup{(\roman*)}]
\item for $\alpha\neq\beta$,
\begin{equation}\label{eq:center-separation}
  |a_n^{(\alpha)}-a_n^{(\beta)}|\to\infty;
\end{equation}
\item $u_n^{(\alpha)}(\cdot+a_n^{(\alpha)})$ is tight,
\begin{equation}\label{eq:profile-convergence}
  u_n^{(\alpha)}(\cdot+a_n^{(\alpha)})\weak\rho^{(\alpha)}
  \quad\text{in }L^p,
  \qquad
  \int_{\R^\d}u_n^{(\alpha)}\to m_\alpha:=\int_{\R^\d}\rho^{(\alpha)};
\end{equation}
\item for every fixed $J$, the pairwise distances among $\supp u_n^{(1)},\ldots,\supp u_n^{(J)},\supp r_n^{(J)}$ tend to infinity as $n\to\infty$;
\item for fixed $J$,
\begin{equation}\label{eq:error-small}
  \norm{e_n^{(J)}}_1\to0;
\end{equation}
\item
\begin{equation}\label{eq:qJ-zero}
  q\bigl[(r_n^{(J)})_n\bigr]\to0\qquad(J\to\infty);
\end{equation}
\item the profiles satisfy
\begin{equation}\label{eq:profile-budgets-density}
  \sum_\alpha m_\alpha
  \leq\liminf_n\int_{\R^\d}\rho_n,
  \qquad
  \sum_\alpha\Lcal(\rho^{(\alpha)})
  \leq\liminf_n\Lcal(\rho_n).
\end{equation}
\end{enumerate}
In \eqref{eq:finite-decomposition}, we call $u_n^{(1)},\ldots,u_n^{(J)}$ the principal pieces, $r_n^{(J)}$ the remainder, and $e_n^{(J)}$ the error term.
\end{proposition}

\begin{proof}
We divide the construction into the extraction of one profile, the choice of a separating annulus, and the iteration of these two steps.

\smallskip\noindent\emph{Extraction of one profile.}
Set $r_n^{(0)}=\rho_n$, and at step $j$ write
\begin{equation}
  \widehat q_j:=q\bigl[(r_n^{(j-1)})_n\bigr].
\end{equation}
If $\widehat q_j=0$, stop.  Otherwise, choose $R_0$ so that
\begin{equation}
  \limsup_{n\to\infty}\sup_{a\in\R^\d}
  \int_{B_{R_0}(a)}r_n^{(j-1)}
  >\frac34\widehat q_j.
\end{equation}
Pass to a further subsequence, retaining all previously obtained convergences, along which the inner spatial supremum in the preceding display is $>\frac23\widehat q_j$.  Choose centers $a_n^{(j)}$ such that
\begin{equation}\label{eq:half-concentration}
  \int_{B_{R_0}(a_n^{(j)})}r_n^{(j-1)}
  \geq\frac12\widehat q_j.
\end{equation}
After a further weak $L^p$ extraction,
\begin{equation}\label{eq:raw-profile-limit}
  r_n^{(j-1)}(\cdot+a_n^{(j)})\weak\rho^{(j)}
  \quad\text{in }L^p.
\end{equation}
Positivity and the uniform $L^1$ bound imply, by testing against an increasing sequence of compactly supported cutoffs, that $\rho^{(j)}\in L^1$ and
\begin{equation}
  \int_{\R^\d}\rho^{(j)}\leq\liminf_n\int_{\R^\d} r_n^{(j-1)}.
\end{equation}
Testing against a cutoff which equals one on $B_{R_0}$ gives
\begin{equation}\label{eq:profile-mass-lower}
  m_j:=\int_{\R^\d}\rho^{(j)}\geq\frac12\widehat q_j>0.
\end{equation}

\smallskip\noindent\emph{Choice of a separating annulus.}
Choose $R_\ell\to\infty$ such that
\begin{equation}
  \int_{|x|>R_\ell}\rho^{(j)}\leq2^{-\ell},
\end{equation}
and set $S_\ell:=R_\ell+\ell$.  For each fixed $\ell$, testing \eqref{eq:raw-profile-limit} against $\one_{B_{R_\ell}}$ and $\one_{B_{S_\ell}\setminus B_{R_\ell}}$ gives
\begin{align}
  \int_{B_{R_\ell}(a_n^{(j)})}r_n^{(j-1)}
  &\longrightarrow\int_{B_{R_\ell}}\rho^{(j)},\\
  \int_{B_{S_\ell}(a_n^{(j)})\setminus
  B_{R_\ell}(a_n^{(j)})}r_n^{(j-1)}
  &\longrightarrow
  \int_{B_{S_\ell}\setminus B_{R_\ell}}\rho^{(j)}.
\end{align}
Choose $(n_\ell)_\ell$ strictly increasing so that, whenever $n\geq n_\ell$, the absolute error in each of these two limits is at most $2^{-\ell}$.  After discarding the finitely many indices $n<n_1$, set $\ell(n):=\max\{\ell:n_\ell\leq n\}$.  This yields radii
\begin{equation}
  R_n^{(j)}:=R_{\ell(n)},
  \qquad
  S_n^{(j)}:=S_{\ell(n)}
\end{equation}
for which
\begin{align}
  \left|
  \int_{B_{R_n^{(j)}}(a_n^{(j)})}r_n^{(j-1)}-m_j
  \right|
  &\leq2^{1-\ell(n)}\to0,
  \label{eq:inner-mass-limit}\\
  0\leq
  \int_{B_{S_n^{(j)}}(a_n^{(j)})\setminus
  B_{R_n^{(j)}}(a_n^{(j)})}r_n^{(j-1)}
  &\leq2^{1-\ell(n)}\to0,
  \label{eq:annulus-error}\\
  S_n^{(j)}-R_n^{(j)}
  &=\ell(n)\to\infty.
\end{align}
\begin{samepage}
\noindent Define
\begin{align}
  u_n^{(j)}
  &:=r_n^{(j-1)}\one_{B_{R_n^{(j)}}(a_n^{(j)})},\\
  e_n^{(j)}
  &:=r_n^{(j-1)}\one_{B_{S_n^{(j)}}(a_n^{(j)})
  \setminus B_{R_n^{(j)}}(a_n^{(j)})},\\
  r_n^{(j)}
  &:=r_n^{(j-1)}\one_{\R^\d\setminus
  B_{S_n^{(j)}}(a_n^{(j)})}.
\end{align}
\end{samepage}
Since $R_n^{(j)}\to\infty$, for every $\varphi\in L^{p'}$, where
$p'=p/(p-1)$,
\begin{equation}
  \left|\int_{\R^\d}(u_n^{(j)}-r_n^{(j-1)})(x+a_n^{(j)})\varphi(x)\diff x\right|
  \leq\norm{r_n^{(j-1)}}_{L^p}\norm{\varphi\one_{B_{R_n^{(j)}}^c}}_{L^{p'}}\longrightarrow0.
\end{equation}
Together with \eqref{eq:raw-profile-limit} and
\eqref{eq:inner-mass-limit}, this proves
\begin{equation}
  u_n^{(j)}(\cdot+a_n^{(j)})\weak\rho^{(j)}
  \quad\text{in }L^p,
  \qquad
  \int_{\R^\d} u_n^{(j)}\longrightarrow m_j,
\end{equation}
which is \eqref{eq:profile-convergence}.
For every fixed $R>0$, testing the weak $L^p$ convergence against $\one_{B_R}$ and then using convergence of the total masses gives
\begin{align}
  \int_{B_R}u_n^{(j)}(x+a_n^{(j)})\diff x
  &\longrightarrow\int_{B_R}\rho^{(j)}(x)\diff x,\\
  \int_{B_R^c}u_n^{(j)}(x+a_n^{(j)})\diff x
  &\longrightarrow\int_{B_R^c}\rho^{(j)}(x)\diff x.
\end{align}
Given $\varepsilon>0$, choose $R$ such that $\int_{B_R^c}\rho^{(j)}(x)\diff x<\varepsilon$.  The second limit bounds the translated tails by $2\varepsilon$ for all sufficiently large $n$; increasing $R$ to handle the finitely many remaining indices proves tightness.

Fix $j<k\leq J$.  Since the remainders are nested,
\begin{equation}
  \supp u_n^{(k)}\cup\supp r_n^{(J)}
  \subset\supp r_n^{(j)}
  \subset\R^\d\setminus B_{S_n^{(j)}}(a_n^{(j)}),
  \qquad
  \supp u_n^{(j)}
  \subset\overline B_{R_n^{(j)}}(a_n^{(j)}).
\end{equation}
Hence,
\begin{align}
  \operatorname{dist}
  \bigl(\supp u_n^{(j)},\supp u_n^{(k)}\bigr)
  &\geq S_n^{(j)}-R_n^{(j)}\longrightarrow\infty,\\
  \operatorname{dist}
  \bigl(\supp u_n^{(j)},\supp r_n^{(J)}\bigr)
  &\geq S_n^{(j)}-R_n^{(j)}\longrightarrow\infty.
\end{align}
Finally, let $R_{0,k}$ denote the fixed radius $R_0$ chosen at the
$k$-th extraction.  Since $r_n^{(k-1)}=0$ on
$B_{S_n^{(j)}}(a_n^{(j)})$, \eqref{eq:half-concentration} gives
\begin{equation}
  0<\frac12\widehat q_k
  \leq\int_{B_{R_{0,k}}(a_n^{(k)})}r_n^{(k-1)}
  =\int_{B_{R_{0,k}}(a_n^{(k)})\setminus
  B_{S_n^{(j)}}(a_n^{(j)})}r_n^{(k-1)}.
\end{equation}
Therefore,
\begin{equation}
  |a_n^{(k)}-a_n^{(j)}|
  \geq S_n^{(j)}-R_{0,k}\longrightarrow\infty,
\end{equation}
which proves the center separation.

\smallskip\noindent\emph{Iteration and the profile inequalities.}
For fixed $J$, set $e_n^{(J)}=\sum_{j=1}^Je_n^{(j)}$.  \Cref{eq:error-small} follows from \eqref{eq:annulus-error}.  For $k<j$, the principal piece $u_n^{(j)}$ is a restriction of the remainder $r_n^{(k)}$ and is therefore supported outside the outer ball $B_{S_n^{(k)}}(a_n^{(k)})$, whereas $u_n^{(k)}$ is supported in the corresponding inner ball $B_{R_n^{(k)}}(a_n^{(k)})$.  Hence, the extracted pieces are pairwise disjoint, and
\begin{equation}
  \sum_{j=1}^J\int_{\R^\d} u_n^{(j)}\leq\int_{\R^\d}\rho_n.
\end{equation}
Passing to the limit shows that the partial sums $\sum_{j=1}^Jm_j$ are uniformly bounded in $J$.  Since $\widehat q_j$ is nonincreasing and $m_j\geq\widehat q_j/2$, it follows that $\widehat q_j\to0$.  On the final diagonal subsequence, the quantity $q$ can only decrease, and hence
\begin{equation}
  q\bigl[(r_n^{(J)})_n\bigr]\leq\widehat q_{J+1}\to0.
\end{equation}

The corresponding local-energy inequality follows similarly from disjointness and weak lower semicontinuity:
\begin{equation}
  \sum_{j=1}^J\Lcal(\rho^{(j)})
  \leq\liminf_n\sum_{j=1}^J\Lcal(u_n^{(j)})
  \leq\liminf_n\Lcal(\rho_n).
\end{equation}
Let $J\to\infty$.  The standard diagonal extraction over the recursive steps completes the construction.
\end{proof}

\subsection{Exact Riesz--Hartree decoupling}

\begin{proposition}[Riesz--Hartree decoupling]\label{prop:Hartree-profile}
Under the assumptions of \Cref{prop:density-profiles}, one may pass to a further subsequence such that
\begin{equation}\label{eq:Hartree-profile}
  \lim_{n\to\infty}D(\rho_n)
  =\sum_{\alpha=1}^\infty D(\rho^{(\alpha)}).
\end{equation}
\end{proposition}

\begin{proof}
The sequence $D(\rho_n)$ is bounded by \eqref{eq:D-mass-L}, so after passing to a subsequence, we may assume it converges.  For fixed $J$, the separated supports, \eqref{eq:error-small}, and the mixed Hardy--Littlewood--Sobolev estimate imply
\begin{equation}\label{eq:finite-Hartree-split}
  D(\rho_n)
  =\sum_{\alpha=1}^JD(u_n^{(\alpha)})
  +D(r_n^{(J)})+o(1).
\end{equation}
Here, all interactions containing $e_n^{(J)}$ vanish as $n\to\infty$ because $\norm{e_n^{(J)}}_1\to0$, its $L^p$ norm is uniformly bounded, and interpolation makes its $L^\sigma$ norm tend to zero, where $\sigma=2\d/(2\d-\s)$.  The interactions between distinct principal pieces, and between a principal piece and the remainder, vanish by support separation and the uniform mass bound.  By \Cref{lem:Hartree-tight},
\begin{equation}\label{eq:u-Hartree-limit}
  D(u_n^{(\alpha)})\to D(\rho^{(\alpha)}).
\end{equation}

It remains to control the remainder uniformly as $J\to\infty$.  Set $q_J:=q\bigl[(r_n^{(J)})_n\bigr]$.  Since $0\leq r_n^{(J)}\leq\rho_n$, the near-field estimate \eqref{eq:vanishing-near} and the far-field bound immediately following that display give, for every $R>0$,
\begin{align}
  \limsup_{n\to\infty}D(r_n^{(J)})
  &\leq C_{\s,\d}
  \left(\limsup_{n\to\infty}\sup_{a\in\R^\d}
  \int_{B_{\frac32\sqrt{\d}\,R}(a)}r_n^{(J)}\right)^{1-\s/\d}
  \sup_n\Lcal(\rho_n)\notag\\
  &\qquad+R^{-\s}\sup_n\norm{\rho_n}_1^2\notag\\
  &\leq C_{\s,\d}
  \left(\sup_n\Lcal(\rho_n)\right)q_J^{1-\s/\d}
  +R^{-\s}\sup_n\norm{\rho_n}_1^2.
\end{align}
Letting $R\to\infty$ and then using $q_J\to0$ as $J\to\infty$ gives $\lim_{J\to\infty}\limsup_{n\to\infty}D(r_n^{(J)})=0$.  Thus, first letting $n\to\infty$ in \eqref{eq:finite-Hartree-split}, and then letting $J\to\infty$, we obtain the desired conclusion.
\end{proof}

\bibliographystyle{amsalpha}
\bibliography{lieb_oxford_fixed_N_riesz_existence_clean_20260817_205957_UTC}

\end{document}